\documentclass[11pt]{amsart}

\usepackage{amsmath}
\usepackage{amssymb}
\usepackage{amsfonts}

\usepackage{amsthm}
\usepackage{enumerate}

\usepackage[mathscr]{eucal}
\usepackage{mathrsfs}

\usepackage{bbm}

\usepackage[svgnames]{xcolor}
\usepackage{subfigure}
\usepackage{tikz}
\usepackage{tikz-3dplot}	
\usetikzlibrary{patterns}

\usepackage{graphicx}
\usepackage{verbatim}

\usepackage[explicit,indentafter]{titlesec}
\titleformat{\section}{\centering\normalfont\scshape}{\thesection.}{.5em}{#1}
\titleformat{\subsection}[runin]{\normalfont\itshape}{\textnormal{\thesubsection.}}{.5em}{#1.}
\titleformat{\subsubsection}[runin]{\normalfont\itshape}{\thesubsubsection.}{.5em}{#1.}
\titlespacing{\section}{0em}{1em}{0.5em}
\titlespacing{\subsection}{0em}{.5em}{0.5em}

\newcommand{\sB}{\mathscr B}

\newcommand{\Be}{\begin{equation}}
\newcommand{\Ee}{\end{equation}}

\newcommand{\Ba}[1]{\begin{array}{#1}}
\newcommand{\Ea}{\end{array}}
\newcommand{\Bea}{\begin{eqnarray}}
\newcommand{\Eea}{\end{eqnarray}}
\newcommand{\Beas}{\begin{eqnarray*}}
\newcommand{\Eeas}{\end{eqnarray*}}
\newcommand{\Benu}{\begin{enumerate}}
\newcommand{\Eenu}{\end{enumerate}}
\newcommand{\Bi}{\begin{itemize}}
\newcommand{\Ei}{\end{itemize}}

\def\a{\alpha}

\def\h{h}

\def\dyad{{\mathrm{dyad}}}

\def\intslash{\rlap{\kern  .32em $\mspace {.5mu}\backslash$ }\int}
\def\qsl{{\rlap{\kern  .32em $\mspace {.5mu}\backslash$ }\int_{Q_x}}}

\newcommand {\Span} {\operatorname{span}}

\def\emph#1{{\it #1 }}

\def\cf{{\it cf}}

\def\loc{{\text{\rm loc}}}

\def\dist{{\text{\it dist}}}

\def\supp{{\text{\rm supp}}}

\def\inn#1#2{\langle#1,#2\rangle}
\def\biginn#1#2{\big\langle#1,#2\big\rangle}

\def\lc{\lesssim}
\def\gc{\gtrsim}

\def\eps{\varepsilon}

\def\la{\lambda}             \def\La{\Lambda}

\def\om{\omega}

\def\fM{{\mathfrak {M}}}

\def\fX{{\mathfrak {X}}}

\def\fZ{{\mathfrak {Z}}}

\def\fa{{\mathfrak {a}}}

\def\fc{{\mathfrak {c}}}

\def\bbC{{\mathbb {C}}}
\def\bbD{{\mathbb {D}}}
\def\bbE{{\mathbb {E}}}

\def\bbN{{\mathbb {N}}}

\def\bbR{{\mathbb {R}}}

\def\bbZ{{\mathbb {Z}}}

\def\sA{{\mathscr {A}}}

\def\sH{{\mathscr {H}}}
\def\sHext{{\mathscr{H}^{\mathrm{ext}}}}

\def\cB{{\mathcal {B}}}

\def\cD{{\mathcal {D}}}
\def\cE{{\mathcal {E}}}

\def\cH{{\mathcal {H}}}

\def\cM{{\mathcal {M}}}
\def\cN{{\mathcal {N}}}

\def\cS{{\mathcal {S}}}

\def\cU{{\mathcal {U}}}
\def\cV{{\mathcal {V}}}

\def\cZ{{\mathcal {Z}}}

 at 10 true pt

\def\be#1{\begin{equation}\label{ #1}}
\def\endeq{\end{equation}}
\def\endal{\end{align}}
\def\bas{\begin{align*}}
\def\eas{\end{align*}}
\def\bi{\begin{itemize}}
\def\ei{\end{itemize}}

\def\eps{\varepsilon}

\def\emph#1{{\it #1}}
\def\textbf#1{{\bf #1}}

\def\bbone{{\mathbbm 1}}

\theoremstyle{plain}
  \newtheorem{theorem}{Theorem}[section]

   \newtheorem{proposition}[theorem]{Proposition}
   
   \newtheorem{lemma}[theorem]{Lemma}
   \newtheorem{corollary}[theorem]{Corollary}

\theoremstyle{remark}
   \newtheorem{remark}[theorem]{Remark}
   
   \newtheorem*{Remark}{Remark}
   \newtheorem*{Remarks}{Remarks}
\theoremstyle{definition}

\numberwithin{equation}{section}

\newcommand {\SE} {{\mathbb E}}
\newcommand {\SN} {{\mathbb N}}
\newcommand {\SR} {{\mathbb R}}

\newcommand {\SZ} {{\mathbb Z}}

\newcommand {\e} {{\varepsilon}}

\newcommand{\al}{{\alpha}}
\newcommand{\dt}{{\delta}}
\newcommand{\Dt}{{\Delta}}

\newcommand {\mand} {{\quad\mbox{and}\quad}}
\renewcommand {\mid} {{\,\,\,\colon\,\,\,}}
\def\supp{\mathop{\rm supp}}
\def\dist{\mathop{\rm dist}}

\newcommand {\Proof} {\noindent{\bf P{\footnotesize\bf ROOF}: } \ }

\newcommand {\hdot}[1]{\langle{#1}\rangle}
\newcommand {\Hdot}[1]{\Bigl\langle{#1}\Bigr\rangle}
\def\lan#1#2{{\langle{#1},{#2}\rangle}}

\newcommand{\hdt}{h^{\dt}_{j,\mu}}
\renewcommand\th{{\widetilde{h}}}

\newcommand{\Bdyad}{{B^{s,\mathrm{dyad}}_{p,q}}}

\newcounter{reg}
\newcommand{\sline}{{\smallskip

\noindent}}

\def\Xint#1{\mathchoice
{\XXint\displaystyle\textstyle{#1}}%
{\XXint\textstyle\scriptstyle{#1}}%
{\XXint\scriptstyle\scriptscriptstyle{#1}}%
{\XXint\scriptscriptstyle\scriptscriptstyle{#1}}%
\!\int}
\def\XXint#1#2#3{{\setbox0=\hbox{$#1{#2#3}{\int}$ }
\vcenter{\hbox{$#2#3$ }}\kern-.6\wd0}}

\def\mint{\Xint-}

\newcommand{\faeven}{{\fa^{\mathrm{even}}}}
\newcommand{\faodd}{{\fa^{\mathrm{odd}}}}

\definecolor{ascol}{rgb}{0,0,1.} 
\definecolor{ggcol}{cmyk}{.74, 0, 1, .41} 
\definecolor{tucol}{rgb}{0.9,.5,0} 

\begin{document}

\title
[Haar frames and dyadic Besov-Sobolev spaces]
{Haar frame characterizations of Besov-Sobolev spaces and  optimal embeddings into their dyadic counterparts}

\author[G. Garrig\'os \ \ \ A. Seeger \ \ \ T. Ullrich] {Gustavo Garrig\'os   \ \ \ \   Andreas Seeger \ \ \ \ Tino Ullrich}

\address{Gustavo Garrig\'os\\ Department of Mathematics\\University of Murcia\\30100 Espinardo\\Murcia, Spain} \email{gustavo.garrigos@um.es}

\address{Andreas Seeger \\ Department of Mathematics \\ University of Wisconsin \\480 Lincoln Drive\\ Madison, WI,53706, USA} \email{seeger@math.wisc.edu}
\address{Tino Ullrich\\ Fakult\"at f\"ur Mathematik\\ Technische Universit\"at  Chemnitz\\09107 Chemnitz, Germany}
\email{tino.ullrich@mathematik.tu-chemnitz.de}
\begin{abstract}
We study the behavior of Haar coefficients in 
Besov and Triebel-Lizorkin spaces on $\SR$, for a parameter range 
in which the Haar system is not an unconditional basis.
First, we obtain a range of parameters, extending up to smoothness $s<1$, in which the spaces $F^s_{p,q}$ and $B^s_{p,q}$ 
are characterized in terms of doubly oversampled Haar coefficients (Haar frames). 
Secondly, in the case that $1/p<s<1$ and $f\in B^s_{p,q}$, we actually prove that the usual Haar coefficient norm, 
$\|\{2^j\langle f, h_{j,\mu}\rangle\}_{j,\mu}\|_{b^s_{p,q}}$
remains equivalent to $\|f\|_{B^s_{p,q}}$, i.e., the classical Besov space is a closed subset of its dyadic counterpart. 
At the endpoint case $s=1$ and $q=\infty$, we show that such an expression gives an equivalent norm for the Sobolev space
 $W^{1}_p(\SR)$, 
$1<p<\infty$, which is related to 
a classical result by Bo\v{c}karev. Finally, in several  endpoint cases  we give optimal inclusions between  $B^s_{p,q}$, $F^s_{p,q}$, and their dyadic counterparts.
\end{abstract}



\maketitle


\section{Introduction and statement of main results}

In this paper we investigate the validity of norm characterizations for elements $f$ in Besov and Triebel-Lizorkin spaces, 
$B^s_{p,q}(\SR)$ and $F^s_{p,q}(\SR)$,
in terms of expressions involving their Haar coefficients 
or suitable variations thereof. The novelty in the current paper is that we  obtain results for a range of the parameters $(s,p,q)$ in which the Haar system is {not}
an unconditional basis of the above spaces (see 
Figures \ref{fig1} and \ref{fig_frame} below); this complements earlier work of the authors \cite{GaSeUl17_2,GaSeUl18,GaSeUl21_2,GaSeUl21, SeeUl17} 
where a complete description was given for the parameter range in which  the unconditional or Schauder basis property holds in each such space. 

We denote the (inhomogeneous) Haar system in $\SR$ by
\Be\label{HaarS}
\sH=\big\{h_{j,\mu}\mid j\geq-1, \,\mu\in\SZ\big\},
\Ee
where we let $h(x)= \bbone_{[0,\frac 12)}(x)-\bbone_{[\frac 12, 1)}(x)$ and
\Be
\label{hjnu}
h_{j,\mu} (x):= h(2^jx-\mu),\quad\text {if $\mu\in\SZ$, $j=0,1,2,\dots$}
\Ee 
Note that $h_{j,\mu}$ is supported in the closure of the dyadic interval
\[
I_{j,\mu}=\big[2^{-j}\mu, 2^{-j}(\mu+1)\big).
\]
In the case $j=-1$, we just let 
\[
h_{-1,\mu}:=\bbone_{I_{-1,\mu}}=\bbone_{[\mu,\mu+1)},\quad \mu\in\SZ.
\]

Let $F^s_{p,q}(\SR)$ and $B^s_{p,q}(\SR)$ denote the usual Triebel-Lizorkin and Besov spaces \cite{Tr83}.
It has been shown in \cite{Tr10,SeeUl17,SeeUl17_2} that $\sH$ is an unconditional basis of $F^s_{p,q}(\bbR)$ if and only if 
$s$ belongs to the range 
\Be\label{eq:uncond-range} \max\Big\{1/p-1,1/q-1\Big\} < s < \min\Big\{1/p,1/q, 1\Big\};\Ee 
moreover, in the range \eqref{eq:uncond-range}, $F^s_{p,q}$ is characterized by
the property
\Be\label{eq:TLcharHaar} 
\Big\|\Big(\sum_{j=-1}^\infty 2^{jsq}\Big[\sum_{\mu\in \bbZ} 2^j |\inn{f}{h_{j,\mu}}|
\bbone_{I_{j,\mu}} \Big]^q\Big)^{1/q}\Big\|_p<\infty,
\Ee
and this expression defines an equivalent quasinorm in $F^s_{p,q}$.

\begin{figure}[h]
 \centering
\subfigure
{\begin{tikzpicture}[scale=2]

\node [right] at (0.75,-0.5) {{\footnotesize unconditional in $F^s_{p,q}$}};

\draw[->] (-0.1,0.0) -- (2.1,0.0) node[right] {$\frac{1}{p}$};
\draw[->] (0.0,-0.0) -- (0.0,1.1) node[above] {$s$};
\draw (0.0,-1.1) -- (0.0,-1.0)  ;

\draw (1.0,0.03) -- (1.0,-0.03) node [below] {$1$};
\draw (2.0,0.03) -- (2.0,-0.03) node [below] {$2$};
\draw (0.03,1.0) -- (-0.03,1.00);
\node [left] at (0,0.9) {$1$};
\draw (0.03,.5) -- (-0.03,.5) node [left] {$\tfrac{1}{q}$};
\draw (0.03,-.5) -- (-0.03,-.5) node [left] {$\tfrac{1}{q}${\small{$-1$}}};
\draw (0.03,-1.0) -- (-0.03,-1.00) node [left] {$-1$};

\draw[dotted] (1.0,0.0) -- (1.0,1.0);
\draw[dotted] (0,1.0) -- (1.0,1.0);
\draw[dotted] (2,0.0) -- (2,1.0);

\path[fill=green!70, opacity=0.4] (0.0,0.0) -- (.5,.5)-- (1.5,0.5) -- (1,0)--(.5,-.5) -- (0,-0.5)--(0,0);
\draw[dotted] (0,0.5)--(1.5,0.5);
\draw[dotted] (0,-0.5)--(0.5,-0.5);

\draw[dashed] (0.0,-1.0) -- (0.0,0.0) -- (1.0,1.0) -- (2,1.0) -- (1.0,0.0) --
(0.0,-1.0);

\end{tikzpicture}
}
\subfigure
{
\begin{tikzpicture}[scale=2]

\node [right] at (0.75,-0.5) {{\footnotesize Schauder basis in $F^s_{p,q}$}};

\draw[->] (-0.1,0.0) -- (2.1,0.0) node[right] {$\frac{1}{p}$};
\draw[->] (0.0,-0.0) -- (0.0,1.1) node[above] {$s$};
\draw (0.0,-1.1) -- (0.0,-1.0)  ;

\draw (1.0,0.03) -- (1.0,-0.03) node [below] {$1$};
\draw (2.0,0.03) -- (2.0,-0.03) node [below] {$2$};
\draw (0.03,1.0) -- (-0.03,1.00);
\node [left] at (0,0.9) {$1$};
\draw (0.03,-1.0) -- (-0.03,-1.00) node [left] {$-1$};

\draw[dotted] (1.0,0.0) -- (1.0,1.0);
\draw[dotted] (0,1.0) -- (1.0,1.0);
\draw[dotted] (2,0.0) -- (2,1.0);

\draw[dashed] (1,0) -- (0.0,-1.0)--(0.0,0.0) -- (1,1)--(2,1);


\draw[white, fill=red!70, opacity=0.4] (0,0) -- (1,1)
-- (2,1.0) -- (1.0,0) --(0,-1)--(0,0);

\draw[very thick,red] (1.98,0.98) -- (1.0,0.0);
\fill[red] (1,0) circle (1pt);
\fill[white] (2,1) circle (1pt);
\draw (2,1) circle (1pt);


\end{tikzpicture}

}
\caption{Parameter domain  for  $\sH$ to be an unconditional basis (left figure) 
or a Schauder basis (right figure) in $F^s_{p,q}(\SR)$. 
}\label{fig1}
\end{figure}
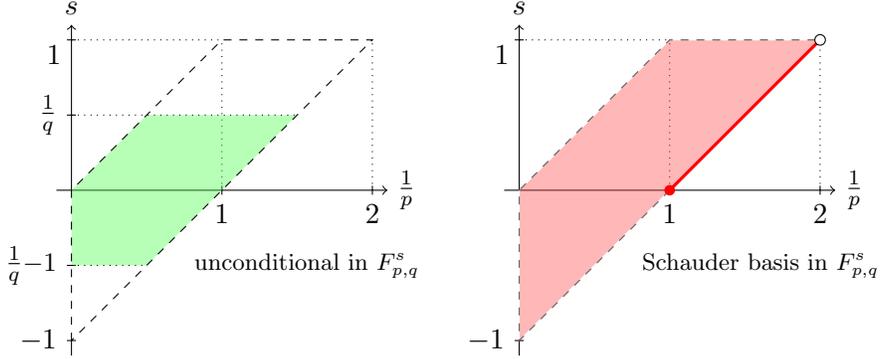

It was also shown   in \cite{GaSeUl18} that $\sH$  is a Schauder basis  of $F^s_{p,q}(\SR)$ 
(with respect to natural enumerations)
 in the larger range
\Be
1/p-1 < s < \min\Big\{1/p, 1\Big\},\quad \mbox{(for all $0<q<\infty$)}.
\label{cond_range}
\Ee
 At the endpoints, the Schauder basis property holds for $F^s_{p,q}$ if and only if 
\Be
s=1/p-1\mand 1/2<p\leq 1
\label{cond_range_end}
\Ee
also for all $0<q<\infty$; see \cite{GaSeUl21_2}.
These regions are depicted in Figure  \ref{fig1}. 

For the spaces $B^s_{p,q}(\SR)$ there is no such distinction, and $\sH$ is an unconditional basis under \eqref{cond_range} (and also a Schauder basis under \eqref{cond_range_end}, if $p=q$); see \cite{GaSeUl21}. Moreover, in the range \eqref{cond_range}, 
$B^s_{p,q}$ is characterized by
the property
\Be\label{eq:BcharHaar} 
\Big(\sum_{j=-1}^\infty 2^{j(s-\frac 1p)q}\Big(\sum_{\mu\in \bbZ} |2^j \inn{f}{h_{j,\mu}}|^p
\Big)^{q/p} \Big)^{1/q}<\infty,
\Ee
and this expression in an equivalent quasinorm in $B^s_{p,q}$.


\subsection{The oversampled Haar systems - Haar frames} 
A main feature of this paper is to show that the above characterizations in terms of Haar coefficients can be extended to the larger regions depicted in Figure \ref{fig_frame} below, provided that  we doubly oversample with Haar type coefficients obtained by a shift.



More concretely, 
we now define 
\Be
\label{thjnu}
\widetilde h_{j,\nu} (x):= 
h(2^jx-\tfrac{\nu}2) 
\quad\text { if $j=0,1,2,\dots$ and $\nu\in\SZ$.}
\Ee
Observe that for even $\nu=2\mu $ we recover the original Haar functions,  $\widetilde h_{j, 2\mu}=h_{j,\mu}$ supported in $I_{j,\mu}$,  but for odd $\nu$ we obtain a \emph{shifted Haar function}
\[
\widetilde h_{j,2\mu+1}=h_{j,\mu} (\cdot - 2^{-j-1}),
\]
which is supported in the 
interval
 $[2^{-j}(\mu+1/2), 2^{-j}(\mu+3/2))=I_{j,\mu}+2^{-j-1}$.
As before, for $j=-1$ we just let 
\[
\widetilde h_{-1,\nu}:=h_{-1,\nu}=\bbone_{[\nu, \nu+1)}.
\]
Then  the \emph{extended Haar system} is defined  by
\Be 
\sHext=  \Big\{\widetilde h_{j,\nu}\mid j\geq-1,\;\nu\in\SZ\Big\}.
\Ee

\

In what follows we will need to work with appropriate spaces of 
distributions on which the (generalized)  Haar coefficients are well defined.
Given a bounded interval $I\subset \SR$, consider the linear functional 
(distribution) 
\Be\label{defoflambdah} 
\la_I(f)=\int_I f(x)dx,
\Ee
which applied to $f\in L_1^\loc$  satisfies the trivial inequality 
\Be\label{eq:trivialL1} |\la_I(f)| \le \int_{I}|f(x)|dx. \Ee

Below, we shall also deal with distributions $f$, associated with certain negative smoothness parameters, which may not belong to $L_1^\loc$. 
To handle these we choose as a  reference space the set of distributions 
\Be
\label{super-space}
\mathscr B:= B^{-1}_{\infty,1}(\bbR).
\Ee 
By standard embeddings, see e.g. \cite[2.7.1]{Tr83}, we have
\Be
B^s_{p,q}(\SR)\hookrightarrow\sB,\quad \mbox{if $s>\frac1p-1$, or $s=\frac1p-1$ and $0<q\leq1$},
\label{BsB}
\Ee
and 
\Be
F^s_{p,q}(\SR)\hookrightarrow\sB,\quad \mbox{if $s>\frac1p-1$, or $s=\frac1p-1$ and $0<p\leq1$}.
\label{FsB}
\Ee
In particular, all the spaces that  are used in Theorems  \ref{mainT-intro}-\ref{thm:neg}  below are embedded into $\sB$.


\begin{proposition}\label{prop:test-funct} \,
For  a bounded interval $I\subset\SR$ consider  the distribution $\la_I$  in \eqref{defoflambdah}. Then
 
(i)  $\la_I$  extends  to a bounded linear functional on $\sB$, with operator norm $O(1+|I|)$. 

(ii) For every  $h\in \sHext$ the linear functional $f\mapsto \inn{f}{h} $ is bounded on $\sB$ with uniformly bounded operator norm.

(iii) If $f\in\sB$ and $\inn{f}{h}=0$, for all $h\in\sH$, then $f=0$.
\end{proposition} 

\begin{Remark}  Clearly,  using \eqref{eq:trivialL1} one can also replace in (i) the space $\sB$ with $\sB+L_1^\loc$. We note that $L_1$ is not embedded into $\sB$, see Remark \ref{R_ex3} below.
\end{Remark}

In the rest of the paper, when $f\in \sB $,  we use the following notation, combining the standard Haar coefficients with the coefficients obtained from the shifted Haar functions:
\Be\label{ext-Haar-coeff} 
\fc_{j,\mu}(f)  =2^j  |\inn{f}{\widetilde h_{j,2\mu}} | + 2^j|\inn{f}{\widetilde h_{j,2\mu+1}} |\, ,
\Ee
when $j=0,1,\dots$, and 
\[
\fc_{-1,\mu}(f) =\inn{f}{h_{-1,\mu}}=\inn{f}{\bbone_{[\mu,\mu+1)}}.
\] 

Our first main result provides a characterization where in \eqref{eq:TLcharHaar} the Haar coefficients $2^j\inn{f}{h_{j,\mu}}$ are replaced with the $\fc_{j,\mu}(f)$.  This covers as well the quasi-Banach range of parameters; 
see Figure \ref{fig_frame} below.

\begin{theorem}\label{mainT-intro} Let $1/2<p<\infty$, $1/2 < q \leq \infty$ and  
\begin{equation}\label{params}
	\max\{1/p-1,1/q-1\} < s < 1\,.
\end{equation}
Then $F^s_{p,q}(\bbR)$ is the collection of all  $f\in \sB$ such that 
\begin{equation}\label{f0}
	\Big\|\Big(\sum\limits_{j=-1}^{\infty} 2^{jsq}\Big|\sum\limits_{\mu\in \bbZ} \fc_{j,\mu}(f)
	\bbone_{I_{j,\mu}}\Big|^q \Big)^{1/q}\Big\|_p < \infty.
\end{equation}
Moreover, the latter quantity represents an equivalent quasi-norm in $F^s_{p,q}(\bbR)$.
\end{theorem}

Using  terminology introduced by  Gr\"ochenig \cite{Gr91}, one may say that 
$ \sHext$ is a 
{\it (quasi-)Banach frame\footnote{In a Hilbert space $H$, 
a frame is a system of vectors $\{e_j\}\subset H$, which for some constants $A,B>0$ satisfies
$
	A\|f\|_H^2 \leq \sum\limits_j |\langle f,e_j\rangle|^2 \leq B\|f\|^2_H\
$, $\forall\;f\in H$.}}
for $F^s_{p,q}(\bbR)$.  
In signal processing language, this can be interpreted by saying that one may stably recover $f$ from the sampled information $\{\langle f,h\rangle: h\in \sHext\}$. 

We remark that the condition  $s>\frac 1q-1$ in \eqref{params} is necessary, in view of the examples 
in \cite{SeeUl17}; see Remark \ref{R_fF} below.
The analogous characterization for  Besov spaces is valid in a larger range:
\begin{theorem}\label{mainB-intro} Let $1/2<p\leq \infty$, $0 < q \leq \infty$ and 
$$
	1/p-1 < s < 1\,.
$$
Then $B^s_{p,q}(\bbR)$ is the collection of all functions $f\in \sB$ such that 
$$
	\Big(\sum\limits_{j=-1}^{\infty} 2^{j(s-1/p)q}\Big(\sum\limits_{\mu\in \bbZ}| \fc_{j,\mu}(f) |^p\Big)^{q/p}\Big)^{1/q} < \infty.
$$
Moreover, the latter quantity represents an equivalent quasi-norm in ${B^s_{p,q}(\SR)}$. 
\end{theorem}

\begin{figure}[h]
 \centering
\subfigure
{\begin{tikzpicture}[scale=2]

\node [right] at (0.75,-0.5) {{\footnotesize $F^s_{p,q}$-spaces}};

\draw[->] (-0.1,0.0) -- (2.1,0.0) node[right] {$\frac{1}{p}$};
\draw[->] (0.0,1.0) -- (0.0,1.1) node[above] {$s$};
\draw (0.0,-1.1) -- (0.0,-1)  ;

\draw (1.0,0.03) -- (1.0,-0.03) node [below] {$1$};
\draw (2.0,0.03) -- (2.0,-0.03) node [below] {$2$};
\draw (0.03,1.0) -- (-0.03,1.00);
\node [left] at (0,0.9) {$1$};
\draw (0.03,.5) -- (-0.03,.5) node [left] {$\tfrac{1}{q}$};
\draw (0.03,-.5) -- (-0.03,-.5) node [left] {$\tfrac{1}{q}${\small{$-1$}}};
\draw (0.03,-1.0) -- (-0.03,-1.00) node [left] {$-1$};

\draw[dotted] (0,0) -- (1.0,1.0);
\draw[dotted] (2,0.0) -- (2,1.0);

\path[fill=green!70, opacity=0.4] (0.0,0.0) -- (0,1)-- (2,1) -- (1,0)--(.5,-.5) -- (0,-0.5)--(0,0);
\draw[dotted] (0,0.5)--(1.5,0.5);
\draw[dotted] (0,-0.5)--(0.5,-0.5);

\draw[dashed] (0.0,-1.0)  -- (2,1.0) -- (0.0,1.0)--(0,-1);

\end{tikzpicture}
}
\subfigure
{
\begin{tikzpicture}[scale=2]

\node [right] at (0.75,-0.5) {{\footnotesize $B^s_{p,q}$-spaces}};

\draw[->] (-0.1,0.0) -- (2.1,0.0) node[right] {$\frac{1}{p}$};
\draw[->] (0.0,-0.0) -- (0.0,1.1) node[above] {$s$};
\draw (0.0,-1.1) -- (0.0,0)  ;

\draw[thick, red] (0.0,-0.97) -- (0.0,0.97);

\draw (1.0,0.03) -- (1.0,-0.03) node [below] {$1$};
\draw (2.0,0.03) -- (2.0,-0.03) node [below] {$2$};
\draw (0.03,1.0) -- (-0.03,1.00);
\node [left] at (0,0.9) {$1$};
\draw (0.03,-1.0) -- (-0.03,-1.00) node [left] {$-1$};

\draw[dotted] (0,0.0) -- (1.0,1.0);
\draw[dotted] (2,0.0) -- (2,1.0);

\draw[dashed, thick]  (0.0,-1.0)--(2,1)--(0,1);


\draw[white, fill=red!70, opacity=0.4] (0,0) -- (0,1)
-- (2,1.0) -- (1.0,0) --(0,-1)--(0,0);

\fill[white] (0,1) circle (1pt);
\draw (0,1) circle (0.6pt);

\fill[white] (0,-1) circle (1pt);
\draw (0,-1) circle (0.6pt);



\end{tikzpicture}

}
\caption{Parameter domain  for  $\sH^{\rm ext}$ to be a characterizing frame for $F^s_{p,q}(\SR)$ (left figure, Theorem \ref{mainT-intro}) and $B^s_{p,q}(\SR)$ (right figure, Theorem \ref{mainB-intro}). 
}\label{fig_frame}
\end{figure}
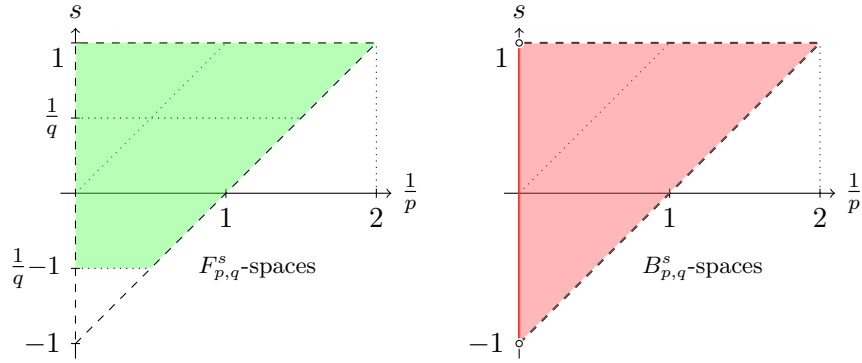

Figure \ref{fig_frame} shows the regions of parameters where  $\sHext$  is a characterizing frame for each of the spaces 
$F^s_{p,q}(\bbR)$ and $B^s_{p,q}(\bbR)$. 

We remark that a related  result, in the special case of the 
H\"older spaces $C^\al=B^\al_{\infty,\infty}(\SR)$, $\al\in(0,1)$,
and using a $1/3$-shifted Haar frame, has been recently obtained by Jaffard and Krim; see
\cite[Theorem \!1]{JaKr21}.
As pointed out by A. Cohen in \cite[Remark 5]{JaKr21}, related characterizations of Besov spaces $B^s_{p,q}[0,1]$, up to smoothness $s<1$, appeared in the spline community in the 70s (see e.g. \cite[\S12.2]{dVL}), in that case in terms of classes of best linear approximation by piecewise constant functions with equally spaced  (or sufficiently mixed) knots. In particular, compare the statements of Theorems \ref{mainB-intro} and \ref{W1-char-intro}, with \cite[Theorem 12.2.4]{dVL} parts (iv) and (ii), respectively; see also Remark \ref{R_sche} below.





\subsection{\it Characterization of $W^{1}_{p}(\bbR)$ via Haar frames} 
We now let  $s=1$, 
and consider in the Banach range  $1\leq p\leq \infty$
the Sobolev space 
$W^{1}_{p} (\bbR)$, endowed with the usual norm
\[
\|f\|_{W^{1}_{p}(\bbR)}=\|f\|_p+\|f'\|_p.
\] 
We also let $BV(\bbR) $ be the subspace of $L_1(\bbR)$ for which the distributional derivative belongs to the space $\cM$ of bounded Borel measures (with the norm given by the total variation of the measure)  and define
\[
\|f\|_{BV(\bbR)}=\|f\|_1+\|f'\|_{\cM}.
\]  
Note that by our definition $BV\subset L_1$ which  deviates from the definition in some other  places in the literature.
We have  the following  result,  that provides characterizations  in terms of  the oversampled Haar system $\sH^{\mathrm{ext}}$.


\begin{theorem}\label{W1-char-intro} For all $f\in \sB$ the following hold.

(i) If $1<p\leq\infty$ then
\[
	\|f\|_{W^1_p} \approx \sup\limits_{j\geq -1} 2^{j(1-1/p)}\Big(\sum\limits_{\mu \in \bbZ} |\fc_{j,\mu} (f)|^p\Big)^{1/p}.
\]

(ii) In the case $p=1$ we have instead 
\[	\|f\|_{BV} \approx \sup\limits_{j\geq -1} \sum\limits_{\mu \in \bbZ} |\fc_{j,\mu} (f)|.
\]
\end{theorem}
Clearly part  (ii) implies the inequality 
\Be \sup\limits_{j\geq -1} \sum\limits_{\mu \in \bbZ} |\fc_{j,\mu} (f)|\lc \|f\|_{W^1_1}\Ee  for all $f\in \cB$. However the converse of this inequality fails as one checks by testing it with  $f=\bbone_{[0,1]} \in BV\setminus W^1_1$; we have  
$ \sup_j\sum_{\mu} |\fc_{j,\mu} (\bbone_{[0,1]})| <\infty$.

The fact that the Sobolev $W^1_p(\SR)$ norm can be expressed in terms of a discrete norm of $b^1_{p,\infty}$ type 
may seem surprising at first, but actually results of this sort can be found in the literature since the 60s, see \cite{Bo69}.
The theorem is also reminiscent of characterizations via the uniform bounds for difference quotients
$h^{-1}(f(\cdot+h)-f) $, see \cite[Prop V.3]{stein-diff} and more recently \cite{bvy, bsvy}.

\subsection{\it Dyadic Besov spaces} 
In this section we present stronger results involving the standard Haar system $\sH$, 
and suitable dyadic variants $B^{s,\mathrm{dyad}}_{p,q} $ of the Besov spaces.

We first recall the definition of the sequence spaces $b^s_{p,q}$ and $f^s_{p,q}$; see \cite{FrJa90}.
If $s\in \bbR$ and $0<p, q\leq \infty$, we define, for $\beta =
 \{\beta_{j,\mu}\}_{j\geq{-1}, \mu \in \bbZ}$,
\begin{equation}\label{b_sequence}
    \|\beta\|_{b^s_{p,q}} :=
    \Big(\sum\limits_{j=-1}^\infty \Big[2^{j(s-1/p)}\Big(\sum\limits_{\mu \in \bbZ} |\beta_{j,\mu}|^p\Big)^{\frac1p}
		\Big]^q\Big)^{\frac1q}, 
\end{equation}
and if $p<\infty$ we let
\begin{equation}\label{f_sequence}
    \|\beta\|_{f^s_{p,q}} := \Big\|\Big(\sum\limits_{j=-1}^\infty\Big|2^{js}\sum\limits_{\mu \in \bbZ} \beta_{j,\mu}
		\bbone_{I_{j,\mu}}(\cdot)\Big|^q\Big)^{1/q}\Big\|_p\,.
\end{equation}
These expressions have the obvious interpretations if $\max\{p,q\}=\infty$.

We additionally define for every $f\in \sB$ the quantity
\[
\|f\|_{B^{s,\mathrm{dyad}}_{p,q}}:=\Big\|\big\{2^j\langle f,h_{j,\mu}\rangle\big\}_{j,\mu}\Big\|_{b^s_{p,q}}
\]
and the vector spaces
\[
B^{s,\mathrm{dyad}}_{p,q}(\bbR) 
=
\Big\{f\in \sB\mid \|f\|_{B^{s,\mathrm{dyad}}_{p,q} }<\infty\Big\}.
\]
Observe that  $\Span \sH\subset B^{s,\mathrm{dyad}}_{p,q}$, so the spaces are not null.
Also, the quantity $\|f\|_{B^{s,\mathrm{dyad}}_{p,q}}$ is a quasi-norm (not just a semi-norm), by Proposition \ref{prop:test-funct}. 
Since $2^j|\inn{f}{h_{j,\mu}} |\le \fc_{j,\mu}(f)$
we note the following immediate consequence of  Theorem \ref{W1-char-intro}.
\begin{corollary}
 If $1\leq p \leq \infty$ then
\Be W^1_p\hookrightarrow B^{1,\dyad}_{p,\infty} .
		\label{BW1}
		\Ee
\end{corollary}

To avoid pathological cases, below we shall typically consider the range
\Be
\frac1p-1 <s < 1,
\label{sp1}
\Ee
and some end-point cases of these. 
We remark that when $s>1$ (or $s=1$ and $0<q<\infty$), the spaces $B^{1,\dyad}_{p,q} $ contain no nontrivial  $C^1$ functions (see Proposition \ref{P_Bdyad1}), 
 while for $s<1/p-1$ the spaces are not complete (see Proposition \ref{P_Bdyad2}). Recall also that in the 
 range $1/p-1<s<\min\{1/p,1\}$
 we have $B^s_{p,q}=B^{s,\dyad}_{p,q}$, cf. \eqref{eq:BcharHaar}. 


Assume now that \eqref{sp1} holds. By Theorem \ref{mainB-intro} we have $B^s_{p,q}\hookrightarrow B^{s,\mathrm{dyad}}_{p,q} $, and 
the inclusion is proper provided that
\Be\label{Haar-dyadic-besov}
1/p<s<1,\quad \text{ or }\; s=1/p\mand \,q<\infty
\Ee 
(since in that range Haar functions do not belong to $B^s_{p,q}$).
Our goal is to prove 
converse inequalities of the form
\Be
\label{BBd}
\|f\|_{B^s_{p,q}}\, \lesssim \, \|f\|_{B^{s,\mathrm{dyad}}_{p,q}} ,\quad \mbox{\emph{provided} that }\; f\in B^s_{p,q}(\SR).
\Ee
Such inequalities will imply that $\|\cdot\|_{B^{s,\mathrm{dyad}}_{p,q}}$ is an equivalent norm in $B^s_{p,q}$,
a result which may seem surprising outside the usual unconditionality region. 
Our first result in this direction is the following.

\begin{theorem} \label{thm:Besov-equiv} Let $1< p \leq \infty$, $0<q\leq \infty$, and $1/p<s<1$.
Then \eqref{BBd} holds. In particular, $B^s_{p,q}$ is a proper closed subspace of $B^{s,\mathrm{dyad}}_{p,q}$, and we have
\Be
	\|f\|_{B^s_{p,q}} \approx  \|f\|_{B^{s,\mathrm{dyad}}_{p,q}}, 
	\text{  for all  
$f \in B^s_{p,q}(\bbR)$.}
\label{B_equiv}
\Ee
\end{theorem}

There are also some precedent results of this nature in the literature. 
When $p=q=\infty$, a norm equivalence as in \eqref{B_equiv} (for continuous functions in the interval [0,1])
was proved by Golubov \cite[Corollary 6]{Go64}; see also 
\cite[Corollary 3.2]{McL69}, \cite[Theorem 7.c.3]{NoSe97} and references therein.


\subsection{\it Inclusions for the limiting case $s=1$}  In what follows the notation $\fX_1\hookrightarrow \fX_2$ will indicate a continuous embedding of the space $\fX_1$ in the space $\fX_2$. As already remarked above we may focus on the cases $s<1$ or $s=1$, $q=\infty$, cf. 
Proposition \ref{P_Bdyad1}.



We now state inclusions into the spaces
$\Bdyad$, in the case that $s=1$ and $q=\infty$. Note that in one of the inclusions we use the smaller space \[
F^{1,\mathrm{dyad}}_{p,\infty}:=\Big\{f\in\sB\mid \big\{2^j\hdot{f,h_{j,\mu}}\big\}_{\substack{j\ge -1\\ \mu\in \bbZ}}\in f^1_{p,\infty}\Big\}.
\]

\begin{theorem}\label{emb_s=1} \phantom{.}  Let $1/2\le p\le \infty$. Then the following hold.

\sline (i) If $1/2\leq p < \infty$ then 
\begin{align} \label{B1pqintoBdyad} 
B^1_{p,q} &\hookrightarrow B^{1,\dyad}_{p,\infty}  \,\iff \, q\le \min\{p,2\}\,,
\\
\label{F1pqintoBdyad} 
F^1_{p,q} &\hookrightarrow B^{1,\dyad}_{p,\infty}  \,\iff \, q\le 2\, .
\end{align} 
For $p=\infty$ we have 
\begin{equation}\label{B1inftyqintoBdyad}
B^1_{\infty,q} \hookrightarrow B^{1,\dyad}_{\infty,\infty}  \,\iff \, q\le 1\,.
\end{equation} 

(ii)  For $1/2<p<\infty$ 
\Be F^1_{p,2} \hookrightarrow F^{1,\dyad}_{p,\infty} .
\Ee
\end{theorem}

The next result gives a converse inequality to the embedding \eqref{BW1}, 
which in particular implies that  
\[
 \|f\|_{B^{1,\mathrm{dyad}}_{p,\infty}} = \sup\limits_{j\geq -1} 2^{j(1-1/p)}
\Big(\sum\limits_{\mu \in \bbZ} |2^j\lan f{h_{j,\mu}}|^p\Big)^{1/p} 
\]
is an equivalent norm in $W^1_p(\SR)$. Earlier bounds of this type, 
for absolutely continuous functions in the interval $[0,1]$, 
can be found in the work of Bo\v ckarev, see  \cite[Theorem 7]{Bo69}, \cite[Theorem I.3.4]{Bo80}, or \cite[Corollary 7.b.2]{NoSe97} 
 and references therein. Below we establish, by different methods, the following result,
which is complementary to Theorem \ref{W1-char-intro}.

\begin{theorem}\label{W1-equiv-intro} Let $1< p\leq \infty$. Then  
\Be
\|f\|_{W^1_p}\,\lesssim\,\|f\|_{B^{1,\mathrm{dyad}}_{p,\infty}}, \quad \mbox{provided} \quad f\in W^1_p(\SR).
\label{W1Bd}
\Ee
In particular, $W^{1}_{p}(\SR)$ is a proper closed subspace of $B_{p,\infty}^{1,\mathrm {dyad}}$, and it holds 
\Be
\|f\|_{W^{1}_{p}}\approx \|f\|_{B^{1,\mathrm{dyad}}_{p,\infty}}\approx \|f\|_{F^{1,\mathrm{dyad}}_{p,\infty}},\quad\text{for }
f\in W^{1}_p (\bbR), \quad 1<p<\infty.
\label{W1BF}
\Ee
\end{theorem}

\begin{Remark}
The inequality in \eqref{W1Bd} (and hence, the first equivalence in \eqref{W1BF}) is also true when $p=1$, 
due to the result of Bo\v{c}karev \cite{Bo69}.
The proof we give here, however, is only valid for $p>1$.
\end{Remark}



\subsection{\it Inclusions for the limiting case $s=1/p-1$} 
We state inclusions into the spaces
$\Bdyad$, in the case that $s=1/p-1$ and $q=\infty$.

\begin{theorem}\label{emb_1/p-1}
(i) For  $0<p, u\leq\infty$ the embedding 
$B^{1/p-1}_{p,u}\hookrightarrow B^{1/p-1, \rm{dyad}}_{p,q}$ can only hold when $q=\infty$. 

(ii) If  $p\ge 1/2$ then 
\Be
B^{1/p-1}_{p,q} \hookrightarrow B^{1/p-1,\mathrm{dyad}}_{p,\infty} \,\iff \, q\le \min\{1,p\} 
\Ee



(iii) If $1/2<p\leq 1$ then 
\Be\label{FinftyinBinfty}
F^{1/p-1}_{p,\infty} \hookrightarrow B^{1/p-1,\mathrm{dyad}}_{p,\infty}.
\Ee
\end{theorem}
\begin{Remark} 
 When $p=1$, we also have the straightforward inequality 
\[\sup_{j\ge -1} \sum_{\mu\in \bbZ} |\inn{f}{h_{j,\mu}}| \lc \|f\|_1,\quad f\in L^1,\] 
which leads to  the inclusion  $L_1\cap \sB\subset  B^{0,  \mathrm{dyad}}_{1,\infty}$.
\end{Remark}




\subsection{\it The case $s=1/p$}
\label{counter_ex}
When $1<p<\infty$, the standard Haar characterization of Besov spaces implies that  
\[
B^s_{p,q}(\mathbb{R})= B^{s,\mathrm{dyad}}_{p,q}(\mathbb{R})\, ,
\quad \frac1p-1<s<\frac1p.
\]
On the other hand, Theorem \ref{thm:Besov-equiv} implies the
norm equivalence
\[
\|f\|_{B^s_{p,q}} \approx  \|f\|_{B^{s,\mathrm{dyad}}_{p,q}} \,, \quad  
f \in B^s_{p,q}, \quad \frac1p<s<1.
\]
These two results might suggest that the norm equivalence could hold also at the dividing line $s=1/p$. 
Here we show that this is not the case, at least when $q=\infty$. 

\begin{theorem}\label{thm:neg} 
Let $1\le  p<\infty$. Then

\sline (i) there exists a sequence $\{f_N\}_{N=1}^\infty$ of functions in $B^{1/p}_{p,\infty} $ such that \[\text{ $\|f_N\|_{B^{1/p,\dyad}_{p,\infty}}  =1$ and $\|f_N\|_{B^{1/p}_{p,\infty} }\gc N$, }
\]

\sline (ii) $B^{1/p,\dyad}_{p,\infty} \setminus B^{1/p}_{p,\infty} \neq \emptyset$.
\end{theorem} 

\begin{Remarks}

\sline (i) The previous result shows that, if $1<p<\infty$, then  Theorem \ref{thm:Besov-equiv} cannot hold at the dividing line $s=1/p$ (at least if $q=\infty$). Namely, 
the embedding of $B^{1/p}_{p,\infty}$ into $B^{1/p,\mathrm{dyad}}_{p,\infty}$, which is established by Theorem \ref{mainB-intro}, is  proper, that is, 
    $B^{1/p}_{p,\infty}\,\subsetneq B^{1/p, \mathrm{dyad} }_{p,\infty} $,
and moreover, on the smaller space $B^{1/p}_{p,\infty}$, the norms are not equivalent, i.e.,
\Be\label{B1p}
\sup\big\{ \|f\|_{B^{1/p}_{p,\infty}} \,: \,  f\in  B^{1/p}_{p,\infty}\text{ and } 
\|f\|_{B^{1/p,\mathrm{dyad}}_{p,\infty}}=1 \big\} 
=\infty\,.
\Ee
Both statements are an immediate consequence of Theorem \ref{thm:neg}.

(ii) Observe that  $BV(\bbR) \hookrightarrow B^1_{1,\infty}(\bbR) $ and that for $p=1$ we have the embedding  $BV(\bbR)\hookrightarrow B^{1,\dyad}_{1,\infty}$ as  a consequence of  Theorem \ref{W1-char-intro}. 
Theorem \ref{thm:neg} shows that this embedding is also  proper, i.e.  
$B^{1,\dyad}_{1,\infty}  \setminus BV(\bbR) \neq \emptyset$.
\end{Remarks}


\subsection{\it Further directions} 
We mention a few problems left open in this paper.
\subsubsection{Besov-type spaces} Concerning \eqref{B1p} in Theorem \ref{thm:neg}, we do not know whether the inequality
\Be
\|f\|_{B^{1/p}_{p,\infty}}\,\leq \,C\,\|f\|_{B^{1/p, \text{dyad} }_{p,\infty}}
\label{B1pgood}
\Ee
 could be true  for $1<p<\infty$ when restricted  to $f\in\cS(\bbR)$.
It is also open to determine whether such inequality
could hold if the $B^{1/p}_{p,\infty}$ norm is replaced by $B^{1/p}_{p,q}$ with $q<\infty$.

\subsubsection{$F^s_{p,q}$ versus $F^{s,\dyad}_{p,q}$ } It would be interesting to establish an optimal analogue of Theorem \ref{thm:Besov-equiv}
for Triebel-Lizorkin spaces.

\subsubsection{Wavelet frames}
The sharp results on the failure of unconditional convergence  of Haar expansions in \cite{SeeUl17} (described  above) have been extended by R. Srivastava  \cite{Srivastava20} to classes of spline wavelets with more restrictive smoothness assumptions. It is then natural  to investigate   extensions of our results on Haar frames to suitable classes of oversampled systems of spline wavelets. 

\subsubsection{Higher dimensions}\label{high_dim}
In this paper we have found appropriate to present our results for function spaces in $\SR$ (rather than $\SR^d$).
It seems likely that many arguments in the paper could be adapted, with only standard modifications, to the higher dimensional context (such as those in \S4 or \S6). On the other hand, the arguments in which the 1-dimensional setting is more present (such as the bootstrapping in \S7) may require more elaborate 
changes.  
We do not pursue these questions here. 

\subsection{\it Structure of the paper} 
In \S\ref{sect:FS} we compile notation and known results about maximal operators in function spaces. We also review some properties about the Chui-Wang wavelet  basis, that will be used later in our proofs. 

In \S\ref{sec:distrspace} we clarify the role of the space $\sB$ and prove Proposition \ref{prop:test-funct}. 

In \S\ref{sec:framescharacterization}  we consider the characterization of function spaces via Haar frames, giving the proofs of Theorems \ref{mainT-intro} and \ref{mainB-intro}  (as a combination of four  Propositions \ref{prop1}, \ref{prop:duality}, \ref{prop4} and \ref{main1}). 

In \S\ref{sec:W1p} we establish Haar frame characterizations of Sobolev and bounded variation spaces and give a proof of Theorem \ref{W1-char-intro}.

In \S\ref{sec:embintoBdyad}
we prove the sufficiency of the conditions for the  embeddings into $B^{1,\dyad}_{p,\infty}$ or $F^{1,\dyad}_{p,\infty} $ in Theorem \ref{emb_s=1}, and the sufficiency for the conditions of embedding into $B^{1/p-1,\dyad}_{p,\infty} $ in Theorem \ref{emb_1/p-1}. 

In \S\ref{sec:normequivalences} 
we prove Theorems
\ref{thm:Besov-equiv} and \ref{W1-equiv-intro}.

In \S\ref{sec-s=1-example} we prove necessary conditions for the embeddings into $B^{1,\dyad}_{p,\infty} $.
Specifically,  in Theorem \ref{emb_s=1} the necessary conditions $q\le p$ in \eqref{B1pqintoBdyad}, $q\le 2$ in \eqref{F1pqintoBdyad}, \eqref{B1pqintoBdyad},  and 
$q\le 1$ in \eqref{B1inftyqintoBdyad}
correspond to Lemma \ref{lem:ex1_B},  Lemma \ref{lem:ex1_F}, 
and  Lemma \ref{lem:q=infty}, respectively.

In \S\ref{sec:s=1/p-1-ex} we obtain  necessary conditions in part (ii) of Theorem \ref{emb_1/p-1}  for the embeddings into $B^{1/p-1,\dyad}_{p,\infty}$.

In \S\ref{sec:neg}
we prove Theorem \ref{thm:neg}.

In \S\ref{S_Bdyad1}
we give a simple proof for the fact that $C^1$ functions in $B^{1,\dyad}_{p,q} $, with $q<\infty$, are constant (Proposition \ref{P_Bdyad1}); moreover show that $B^{s,\dyad}_{p,q}$ is not complete when $s<1/p-1$ (Proposition \ref{P_Bdyad2}) and finally prove 
 part (i) of Theorem \ref{emb_1/p-1}  (see Proposition \ref{P_ex3}).

\subsection*{Acknowledgements} The authors would like to thank two anonymous referees for valuable comments, which have been taken into account in Remarks \ref{R_sche} and \ref{high_dim}. The authors thank Hans Triebel for his comments on duality, see Remark \ref{triebel_rem}. In addition, T.U.  would like to thank Dorothee Haroske and Winfried Sickel for having the opportunity to present the material in a plenary talk at the conference ``Function spaces and applications'' in Apolda 2022.  
G.G. was supported in part by grant
PID2019-105599GB-
I00 from Ministerio de Ciencia e Innovación (Spain), and grant 20906/PI/18 from Fundaci\'on S\'eneca (Regi\'on de Murcia, Spain). A.S. was supported in part by National Science Foundation grant DMS 2054220. T.U. was supported in part by  Deutsche Forschungsgemeinschaft (DFG), grant 403/2-1.

\section{Preliminaries on function spaces and wavelet bases} 
\label{sect:FS}

 \subsection{\it Definition of spaces}
\label{sect:MF_FS} 

Let $s\in\SR$ and $0<p, q\leq \infty$ be given. 
We shall use both definitions and characterizations of $B^s_{p,q}$ and $F^s_{p,q}$ in terms of dyadic frequency decompositions and in terms of sequences of compactly supported  kernels with cancellation (see e.g. 
\cite[2.5.3, 2.4.6]{Tr92} or \cite[\S1.3,1.4]{Tr06} where the terminology local means is used). 

Consider two functions $\beta_0, \beta \in \cS(\SR)$ such that
 $|\widehat{\beta}_0(\xi)|>0$ when  $|\xi|\leq1$ and 
 $|\widehat{\beta}(\xi)|>0$ when $1/4\leq|\xi|\leq1$. Assume further that $\beta(\cdot)$ has  vanishing
moments up to a sufficiently large order $M\in \bbN$, that is, 
\Be
\label{betacanc}
\int_{\bbR} \beta(x)\,x^{m}\,dx = 0\quad\mbox{when}\quad m < M\,.
\Ee 
The precise value of $M$ is not relevant, but for the properties used in the paper it will suffice with  
\Be 
\label{Mcondition} M >1/p+|s|+2 . 
\Ee 
We let  $\beta_k(x):=2^{k}\beta(2^kx)$, $k\geq1$, and define for $k\in\SN_0$ the convolution operators 
\[L_kf:=\beta_k*f,\]
acting on distributions $f\in \cS'(\SR)$. 

The Besov space $B^s_{p,q}(\bbR)$ is the set of all distributions $f\in \cS'(\bbR)$ such that 
\Be
\label{localmeans_B} 
    \|f\|_{B^s_{p,q}} := \Big(\sum_{k=0}^\infty \Big(2^{ks}\|L_k f\|_p\Big)^q\Big)^{1/q} <\infty.
    \Ee 
If $p<\infty$, the Triebel-Lizorkin space $F^s_{p,q}(\bbR)$ is the set of all $f\in \cS'(\bbR)$ such that 
\Be\label{localmeans_F} 
    \|f\|_{F^s_{p,q}} := \Big\|\Big(\sum_{k=0}^\infty 2^{ksq}|L_k f(x)|^q\Big)^{1/q}\Big\|_p<\infty.
    \Ee
Different choices of $\beta_0,\beta$ give rise to the same spaces and equivalent quasi-norms; see e.g. \cite[Theorem 1.7]{Tr06}. 
From now on we will assume that  
\[
\supp\beta_0\subset(-1/2,1/2)\mand \supp\beta\subset(-1/2,1/2).
\]

We shall often use the following decomposition of distributions in $\cS'(\SR)$.
Let $\eta_0\in C^\infty_0(\bbR)$ be supported on $\{|\xi|<3/4\}$ and such that $\eta_0(\xi)=1$ when $|\xi|\leq1/2$. 
Define the convolution operators $\La_0$, and $\La_k$ for $k\ge 1$, by
\begin{align*}
\widehat{\La_0 f}(\xi) &=\frac{\eta_0(\xi)}{\widehat \beta_0(\xi)}\widehat f(\xi)
\\
\widehat{\La_k f}(\xi) &=\frac{\eta_0(2^{-k}\xi) -\eta_0(2^{-k+1}\xi)}{\widehat \beta(2^{-k}\xi)}\widehat f(\xi), \quad k\ge 1.
\end{align*}
Then, for all $f\in\cS'(\SR)$ we have
\Be
f=\sum_{k=0}^\infty L_k \La_k f
\label{LLak}
\Ee
 with convergence in $\cS'(\SR)$. 
Also, 
it holds
\[
\Big(\sum_{k=0}^\infty \Big(2^{ks}\|\La_k f\|_p\Big)^q\Big)^{1/q}\lesssim \|f\|_{B^s_{p,q}},
\]  
 and likewise for the $F$-norms.

\subsection{\it Maximal functions}

We follow Triebel \cite{Tr83,Tr92}. 
Given $f\in L_1^{\rm loc}(\SR)$, consider the Hardy-Littlewood maximal function, defined by 
\begin{equation}\label{HLmax}
    Mf(x) := \sup\limits_{x \in I} \frac{1}{|I|}\int_I |f(x)|\,dx\,,
\end{equation}
where the $\sup$ is taken over all intervals $I$ that contain $x$. 
A classical result of Fefferman and Stein asserts that, if $1<p<\infty$ and $1<q\leq\infty$ then
\begin{equation}\label{HLmaxineq}
    \Big\|\Big(\sum\limits_{j} |Mf_j|^q\Big)^{1/q}\Big\|_p \lesssim \Big\|\Big(\sum\limits_{j} |f_j|^q\Big)^{1/q}\Big\|_p\,.
\end{equation}
for all sequences of measurable functions $\{f_j\}$ with finite right hand side.

Let us further define the Peetre maximal functions \cite{Pe75}. Given $j\in \bbN$ and $A>0$ we let  
 \[\fM^{**}_{j,A} f(x)= \sup_{h\in \bbR}  \frac{|f(x+h)|}{(1+2^j|h|)^A} .\]
 Let $\cE_j$ be the set of distributions $f\in \cS'(\bbR)$ such that $\supp \widehat f$ is supported 
in an interval of diameter $\le 2^{j+2}$. Then for all $f\in\cE_j$ it holds
\Be
\fM^{**}_{j,A} f(x) \lesssim_{s, A}\,\big[M(|f|^s)(x)\big]^{1/s},
\label{MPee}
\Ee
provided that $s\geq 1/A$; see \cite{Pe75} or \cite[Theorem 1.3.1]{Tr83}. 
In particular, if  $0<p\leq \infty$ and $A>1/p$ then
 \Be
 \|\fM^{**}_{j,A} f\|_p \le C_{p,A} \|f\|_p, \quad f\in\cE_j.
 \label{Peetremax1}
\Ee
 Also, from \eqref{HLmaxineq} and \eqref{MPee}, if $0<p<\infty$, $0<q\leq \infty$ and $A>\max\{ 1/p, 1/q\} $, then
 \begin{equation}\label{Peetremax}
    \Big\| \Big(\sum_{j} |\fM^{**}_{j,A} f_j|^q\Big)^{1/q} \Big\|_p 
 \le C_{p,q,A} 
 \Big\| \Big(\sum_{j} | f_j|^q\Big)^{1/q} \Big\|_p 
 \end{equation} 
for all sequences of functions $(f_j)$ such that $f_j\in \cE_j$.

Below we shall also use the (smaller) maximal functions
\Be\label{eq:fMk*}
\fM_jf(x)=\sup_{|h|\leq 2^{-j}}|f(x+h)|\mand \fM^*_jf(x)=\sup_{|h|\leq 2^{-j+2}}|f(x+h)|.
\Ee 
Note that for all $A>0$ it holds 
\[
\fM_jf(x)\leq \fM^*_jf(x)\lesssim \fM^{**}_{j,A}f(x),
\]
so in particular, for all $0<p\leq\infty$ we have
\Be
\big\|\fM^*_jf\|_p\lesssim\|f\|_p,\quad f\in\cE_j.
\label{fM1}
\Ee
We shall also make use of the following elementary inequality: if $0<p\leq\infty$ then
\Be
\label{5.8}
\|\fM^*_j f\|_p\lesssim 2^{\ell/p}\,\|\fM_{j+\ell}f\|_p,\quad \ell\geq0.
\Ee
To prove this assertion, if we let $x_\nu=\nu 2^{-(j+\ell)}$, then we have 
\[
\fM^*_jf(x)\leq \sup_{|\nu|\leq 2^{\ell+2}}|\fM_{j+\ell}f(x+x_\nu)|\leq
 \Big(\sum_{|\nu|\leq 2^{\ell+2}}|\fM_{j+\ell}f(x+x_\nu)|^p\Big)^\frac1p.
\]
Then, taking $L_p$ quasi-norms one easily obtains \eqref{5.8}. 

 \subsection{\it Chui-Wang wavelets}\label{S_chui}
 The proofs of Theorems \ref{mainT-intro} and \ref{mainB-intro}  will require a characterization in terms of a wavelet basis generated by the {\it  Chui-Wang polygon}  and its dual. 
 
 Define the $m$-fold convolution $\cN_m$ of characteristic functions of $[0,1)$, i.e. 
 $\cN_1=\bbone_{[0,1)}$, and, for $m\ge 2$,  $\cN_m=\cN_{m-1} *\bbone_{[0,1)}$.
 In particular we get for $m=2$ the 
  \emph{hat function} 
 \Be \label{eq:CW-N2} \cN_2(x)= \begin{cases} x,  &\quad  \text{ $x\in [0,1]$,}
 \\2-x, &\quad \text{ $x\in [1,2]$,} 
 \\
 0, &\quad \text{ $x\in \bbR\setminus[0,2]$. }
 \end{cases}\Ee
Let
\[
\cN_{2;j,\nu}(x):=\cN_2(2^{j}x-\nu),\quad j\geq0, \;\nu\in\SZ,
\]
which is a hat function adapted to $\supp \cN_{2;j,\nu}=[2^{-j}\nu, 2^{-j}(\nu+2)]$. 

The next elementary observation will be crucial in what follows.

\begin{lemma}
\label{L_Nder}
If $f$ is locally absolutely continuous in $\SR$, then for all $j\geq1$ and $\nu\in\SZ$ it holds
\Be
\label{Nder}
\lan {f'}{\cN_{2;j,\nu}} \,=\,-2^j\,\lan{f}{\th_{j-1,\nu}}, 
\Ee
while for $j=0$ it holds 
\[
\lan {f'}{\cN_{2;0,\nu}}=-\lan f
{h_{-1,\nu}}+\lan f{h_{-1,\nu+1}},\quad \nu\in\SZ.
\]
\end{lemma}
\begin{proof}
Integrating by parts one has
\[
\lan {f'}{\cN_{2;j,\nu}} \,=\,-\lan f {(\cN_{2;j,\nu})'}.
\]
Now, a simple computation gives
\Beas
(\cN_{2;j,\nu})' & = & 2^j\bbone_{(2^{-j}\nu, 2^{-j}(\nu+1))}-2^j\bbone_{(2^{-j}(\nu+1), 2^{-j}(\nu+2))}\\
& = & 2^j\bbone_{I_{j,\nu}}-2^j\bbone_{I_{j,\nu+1}} \,=\,2^j\,\th_{j-1,\nu},
\Eeas
where the last equality follows from the definition of the shifted Haar system in \eqref{thjnu} (if $j\geq1$).
Combining the two expressions one obtains \eqref{Nder}. The case $j=0$ is similar.
\end{proof}

The Chui-Wang polygon \cite[Theorem 1]{ChuiWang92}, \cite[6.2.5, 6.2.6]{Chui92} is the compactly supported wavelet given by 
\begin{equation} \label{bk-coefficients} 
  \begin{split}
    \psi(x) &= \frac{1}{2}\sum\limits_{\ell = 0}^2(-1)^{\ell}\cN_4(\ell+1)\mathcal{N}_{4}^{\,''}(2x-\ell)\,,\\
    &= \frac{1}{2}\sum\limits_{\ell = 0}^2(-1)^{\ell}\mathcal{N}_4(\ell+1)\sum\limits_{j=0}^2(-1)^j\binom{2}{j}\mathcal{N}_2(2x-j-\ell)\\
    &= \sum\limits_{k\in \bbZ} b_k\cdot \mathcal{N}_2(2x-k)
  \end{split}
\end{equation}
where $(b_k)$ is a finite sequence supported in $\{0,\ldots,4\}$.
 The wavelet $\psi$ is compactly supported and has two vanishing moments, i.e., 
$\int \psi(x)  dx = \int x\psi(x)dx = 0$. For $j\in \bbN_0$ and $\mu\in \bbZ$ let 
\[\psi_{j,\mu} (x)=\psi(2^j x-\mu), \]
while for $j=-1$ we let
\[
\psi_{-1,\mu}(x)=\cN_{2;0,\mu}(x)=\cN_2(x-\mu).
\]
 Then we have the orthogonality relations  with respect to different scales
$$
	\langle \psi_{j,\mu}, \psi_{j',\mu'}  \rangle = 0\quad,\quad j \neq j'\,.
$$ 

In contrast to that it only forms a Riesz basis within one and the same scale with respect to different translations. 
The dual basis can be computed precisely \cite{DeUl20} and does not provide compact support. 
\begin{figure}[h]
 	\includegraphics[scale = 0.45]{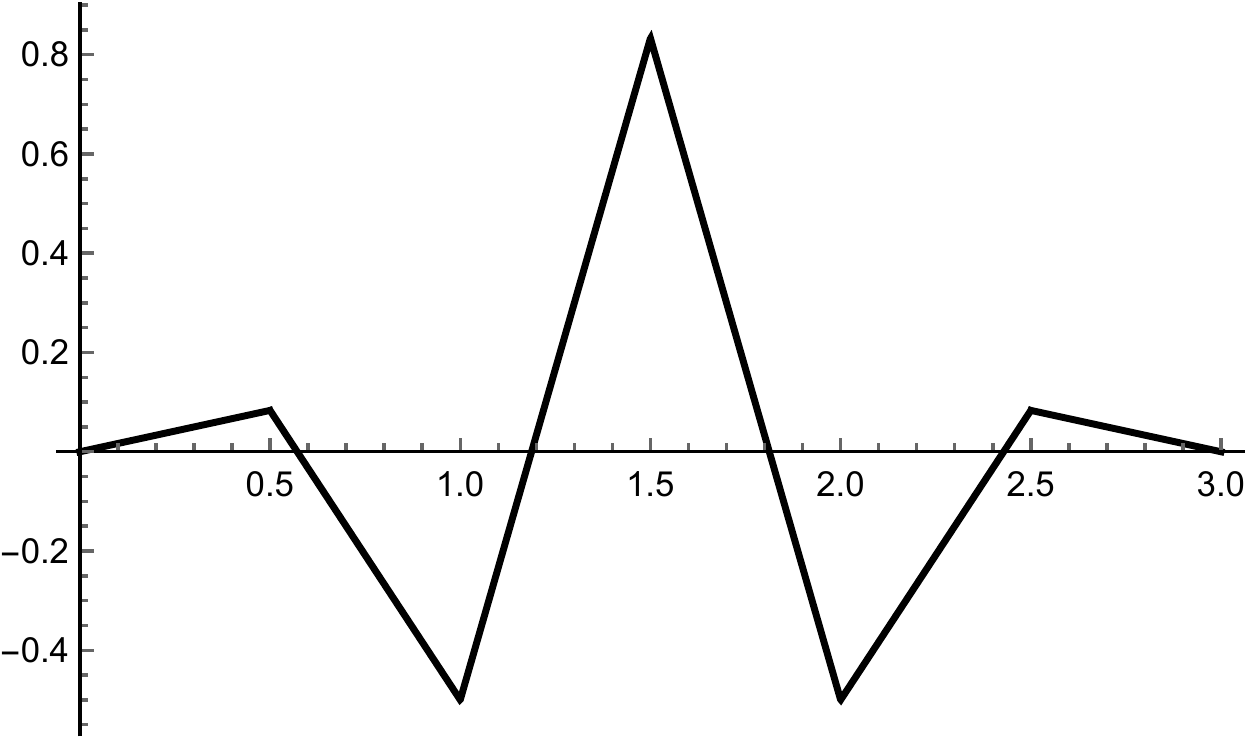}
 	\includegraphics[scale = 0.45]{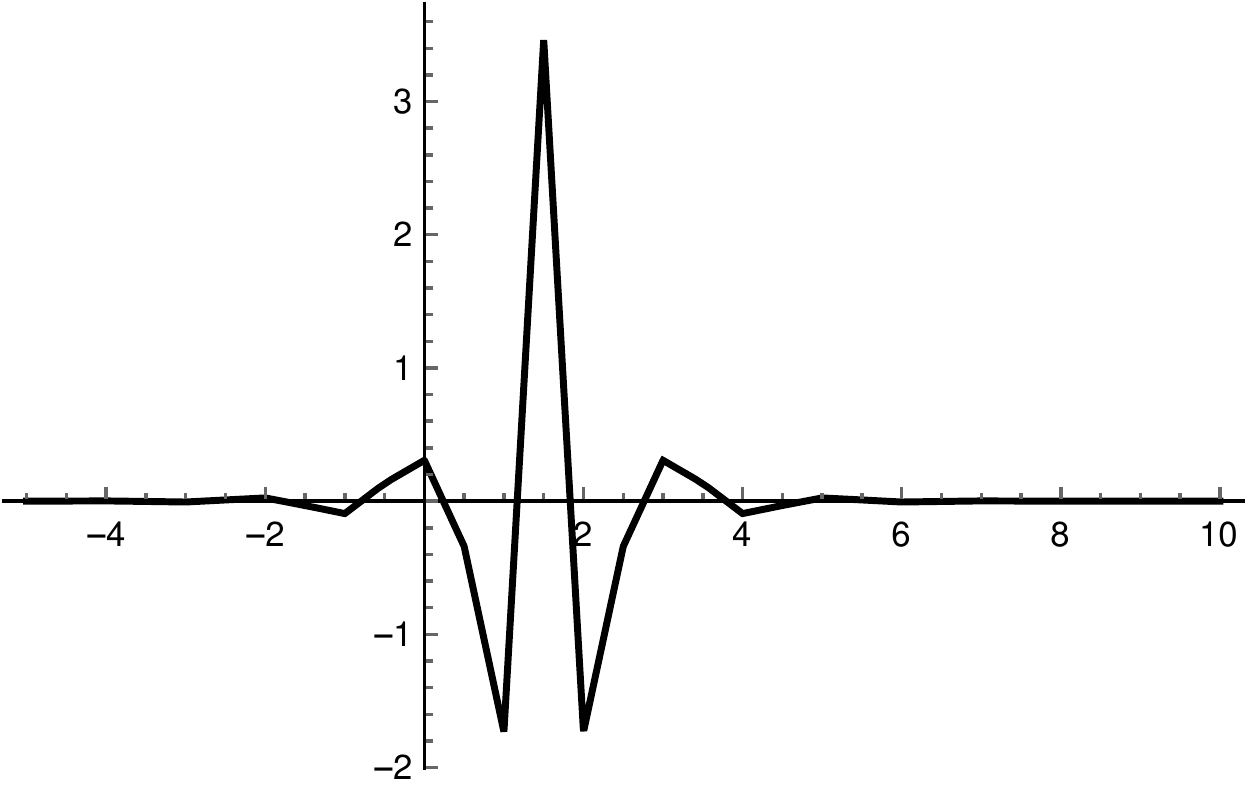}
 	\caption{The Chui-Wang wavelet $\psi_2$ of order $2$ and its dual $\psi^*_2$}
\end{figure}
However, the coefficients $a_k$ in 
\begin{equation}\label{f13}
	\psi^*(x) = \sum\limits_{k\in \bbZ} a_k\psi(x-k)
\end{equation}
are exponentially decaying; see the paper \cite{DeUl20} for explicit formulas for $a_k$.
Observe from \eqref{f13} that also $\psi^\ast$ has two vanishing moments. 

Using this construction, Derevianko and Ullrich provided the following characterization for
the $F^s_{p,q}$ and $B^s_{p,q}$ spaces; see \cite[Theorem 5.1]{DeUl20}.

\begin{theorem}\cite{DeUl20} \label{thm:Chui-Wang} Let $0<p\leq \infty$, $0<q\leq \infty$. 

(i) If $p<\infty$ and \Be \label{Chui-1} \max\{1/p,1/q\}-2 < r < 1+\min\{1/p,1/q, 1\}\Ee then we have for all $f \in B^{-2}_{\infty,1}(\bbR)$  
\begin{equation}\label{f0_C}        
    \|f\|_{F^r_{p,q}}\approx\Big\|\Big(\sum\limits_{j=-1}^{\infty} 2^{jrq}\Big|\sum\limits_{\mu\in \bbZ} 
	2^j\langle f,\psi_{j,\mu}\rangle\bbone_{I_{j,\mu}}\Big|^q \Big)^{1/q}\Big\|_p.
\end{equation}

(ii) If 
\Be \label{Chui2} 1/p-2<r<\max\{1+1/p,2\} \Ee  then we have for all $f\in B^{-2}_{\infty,1}(\bbR)$
\begin{equation}\label{f2_C}
	\|f\|_{B^r_{p,q}} \approx
	\Big(\sum\limits_{j=-1}^{\infty} 2^{j(r-1/p) q}\Big[\sum\limits_{\mu\in \bbZ} 
	|2^j\langle f,\psi_{j,\mu}\rangle|^p\Big]^{q/p} \Big)^{1/q} \,.
\end{equation}
\end{theorem}

\begin{remark}\label{R_DU}
Concerning part (i), we remark that the result stated in \cite[Theorem 5.1]{DeUl20}, requires the additional restriction $r<1$, 
which comes from a similar restriction in \cite[Proposition 5.4]{DeUl20}. This restriction, however, can be lifted and replaced
by $r<1+\max\{1/p,1/q\}$, using a complex interpolation argument which involves part (ii) (case $p=q$), as we discuss 
in Step 3 of Proposition \ref{prop4} below.
\end{remark}

\section{Haar functions as linear functionals on $\sB$:\\Proof of Proposition \ref{prop:test-funct}} 
\label{sec:distrspace}
\subsection{\it Proof of (i) and (ii)} \label{SS_3.1}
Since every $h\in \sH$ is a difference of two characteristic functions of intervals of length $\le 1$ part  (ii) is an immediate consequence of part (i). It suffices to analyze  $\la_I$ on $\sB$,  for each bounded interval $I$.
Let $f\in \sB=B^{-1}_{\infty,1}(\SR)$. Using the decomposition in \eqref{LLak}
we can write $f=\sum_{k=0}^\infty  L_k f_k$ where the Fourier transform  $\widehat f_k$ is supported in $\{\xi:|\xi|\le 2^k\} $ and the $f_k$ satisfy
\Be\label{eq:Bspace-dec}\sum_{k\ge 0} 2^{-k} \|f_k\|_\infty \lc \|f\|_{\sB}; \Ee here,  $L_k f=\beta_k*f$
where $\beta_0,\beta$ are even functions in $C^\infty_c(-1/2,1/2)$, and  $\beta_k=2^{k}\beta(2^{k}\cdot)$,  for $k\ge 1$. 
Also, $\int\beta (x)dx=0$. 

We let $\lambda_I(f):=\sum_{k=0}^\infty\lan{f_k}{L_k\bbone_I}$, which we sometimes denote by $\lan{f}{\bbone_I}$.
Note that 
\Be\label{eq:elemk=0} \|L_0\bbone_I\|_1\lc |I|.\Ee By \eqref{eq:Bspace-dec} one needs   to show that $\|L_k  \bbone_I \|_1\lc  2^{-k}$ for $k\ge 1$;  one actually gets the better estimate
\Be\label{eq:elem} 
\|L_k  \bbone_I \|_1\lc  \min\{ |I|, 2^{-k}\}, \quad k\ge 1.
\Ee
Hence $|\lambda_I(f)|\lc \max\{|I|,1\} \sum_{k\ge 0} 2^{-k} \|f_k\|_\infty$ and we  deduce that $\la_I$ extends to a  bounded  linear functional on $\sB$, with $\|\la_I\|_{\sB^*}\lc \max\{1,|I|\}$.

It remains to verify  \eqref{eq:elem}. Fix $I$, with center $y_I$, 
and assume first that $2^{-k}>|I|$. Then the function 
$L_k \bbone_I=\beta_k*\bbone_I$ is supported in an interval centered at $y_I$ with length $O(2^{-k})$ 
and satisfies
$|\beta_k *\bbone_I(x)|\leq \|\beta_k\|_\infty\,\|\bbone_I\|_1
\lesssim 2^{k} |I|.$
Thus we obtain
$\|\beta_k *\bbone_I\|_1\lesssim |I| $ which is in \eqref{eq:elem} in this case.

Now assume $2^{-k}\le |I|$. Let $y_+$ and $y_-$ be the endpoints of $I$.
 Let $\cU_k$ be the union of the two closed  intervals of length $2^{-k+2}$ 
centered at $y_+$ and $y_-$.
Then $\beta_k*\bbone_I $ is supported in $\cU_k$,
which has size  $|\cU_k|=O(2^{-k})$. This assertion, combined with 
$\|\beta_k *\bbone_I \|_\infty \leq \|\beta_k\|_1\,=O(1)$, also implies
\eqref{eq:elem} in this case.

\subsection{\it Proof of (iii)} For the argument below we shall use the dyadic averaging operators,
defined for $N\geq0$ by
\Be
\label{expect}
\bbE_N f(x):=\sum_{\mu \in \bbZ} \,2^{N} \,\langle f, \bbone_{I_{N,\mu }}\rangle\,\bbone_{I_{N,\mu }}(x).
\Ee
In view of (i), these operators can be defined acting
on distributions $f\in\sB$ such that $E_N:\sB\to L_\infty$ has operator norm $O(2^N)$.

Let now $f\in\sB$ such that $\inn fh=0$, for all $h\in\sH$. Since each $\bbone_{I_{N,\mu}}$ belongs to $\Span\cH$, 
this implies that $\SE_Nf=0$, for all $N\geq0$. 
We then must show that $f=0$, which is a direct consequence of part (b) in the following lemma.

\begin{lemma}
\label{L_sB}
(a) The operators $\SE_N$ satisfy the uniform bound

\[
\sup_{N\geq0}\big\|\SE_N\|_{\sB 
\to B^{-1}_{\infty,\infty}}<\infty.
\]

(b) If $f\in \sB 
$\; then $\;\|\SE_Nf -f \|_{B^{-1}_{\infty,\infty}}\to0$ as $N\to\infty$.
\end{lemma}

\begin{proof}

Part (a) is implicit in \cite{GaSeUl21}. Indeed, it follows by combining the estimates stated 
in the four propositions in \cite[\S4]{GaSeUl21}, for the cases $s=-1$ and $p=\infty$.

We now show part (b). Let $f\in\sB$ and write $f=\sum_{k=0}^\infty L_k f_k$ as at the beginning of \S\ref{SS_3.1}. The above series converges in $\cS'$ and also in the $\sB$-norm. Then, given $\e>0$ one can find $g=\sum_{k=0}^JL_kf_k$ such that
\[
\|f-g\|_\sB<\e.
\]
Observe that $g$ is bounded, since
\[
\|g\|_\infty\leq \sum_{k=0}^J\|L_kf_k\|_\infty\lesssim 
\sum_{k=0}^J \|f_k\|_\infty\lesssim 2^J\|f\|_\sB.
\]
A similar reasoning shows that $\|g'\|_\infty<\infty$, so in particular $g$ is uniformly continuous.
Thus exists an integer $N_0\in\SN$ such that
\[
\|g-\SE_Ng\|_\infty<\e, \text{ for all }\,N\geq N_0.
\]
Combining these assertions, and using the trivial embeddings \[
L_\infty\hookrightarrow \sB\hookrightarrow B^{-1}_{\infty,\infty},
\]
we obtain, for all $N\geq N_0$,
\Beas
\|f-\SE_Nf\|_{B^{-1}_{\infty,\infty}} & \leq & 
\|f-g\|_{\sB}
+\|g-\SE_Ng\|_\infty + 
\|\SE_N(g-f)\|_{B^{-1}_{\infty,\infty}}\\
& \lesssim & 2\e+ C\,\|g-f\|_\sB\lesssim \e,
\Eeas
where in the second inequality we have used part (a).
\end{proof}
\begin{remark}\label{triebel_rem} As H. Triebel pointed out to us, it is possible to give a different proof of Proposition \ref{prop:test-funct} based on duality identities as in
\cite[Remark 2, p.180]{Tr83}.
To do so, one can regard $\lan{f}{\bbone_I}$ as a duality pairing using the facts $\bbone_I\in B^1_{1,\infty}=(\mathring{B}^{-1}_{\infty,1})^*$, and $f\in \mathring{B}^{-1}_{\infty,1}$ whenever $f\in\sB$ with compact support.
\end{remark}

\section{Characterizations by Haar frames:\\Proofs of Theorems \ref{mainT-intro} and \ref{mainB-intro}}
\label{sec:framescharacterization}

The proofs  of Theorems \ref{mainT-intro} and \ref{mainB-intro}  will follow from  the four Propositions 
\ref{prop1}, \ref{prop:duality}, \ref{prop4} and \ref{main1} stated below.

The first proposition is a strengthening of \cite[Proposition 2.8]{Tr10}. It gives one of the inclusions asserted in 
Theorems \ref{mainT-intro} and \ref{mainB-intro}. The  region of indices is the same as in Figure \ref{fig_frame}. We set $\fc(f)=\{\fc_{j,\mu}(f)\}_{j\ge -1, \mu\in \bbZ}$
with $\fc_{j,\mu}(f)$ as in \eqref{ext-Haar-coeff}.

\begin{proposition}\label{prop1} Let $0< p, q \leq \infty$ and $s\in\SR$.

(i) If $p<\infty$ and $\max\{1/p-1,1/q-1\} < s < 1\,$,
then for all $f\in F^s_{p,q}$   
\Be \|\fc(f)\|_{f^s_{p,q}}
	\lesssim \|f\|_{F^s_{p,q}}\,.
	\label{fF1}
	\Ee
(ii) If $1/p-1 <s<1$, then
for all $f\in B^s_{p,q}$
\[ \|\fc(f)\|_{b^s_{p,q}}
 \lc \|f\|_{B^s_{p,q}}.	
\]
\end{proposition}

\begin{proof} 
To avoid dealing separately with $\widetilde{h}_{j,\nu}$ with $\nu$ even or odd, we prove a slightly more general result.
For a fixed $\dt\in[0,1]$ and for $j\geq0$ and $\mu\in\SZ$, consider the shifted Haar function
\[
h^{\dt}_{j,\mu}(x):=h_{j,\mu}(x-\dt2^{-j})=h\big(2^jx-(\mu+\dt)\big),
\]
whose support is the  interval $I_{j,\mu}+\dt2^{-j}$. 
When $j=-1$, we just let $h^{\dt}_{-1,\mu}=h_{-1,\mu}=\bbone_{[\mu,\mu+1)}$.
Part (i) will then be a consequence
of the following estimate
\Be
\Big\|\Big(\sum\limits_{j=-1}^{\infty} 2^{jsq}\Big|\sum\limits_{\mu\in \bbZ} 
	2^j\langle f, \hdt\rangle\;
	\bbone_{I_{j,\mu}}\Big|^q \Big)^{1/q}\Big\|_p 
	\lesssim \|f\|_{F^s_{p,q}}\,,
\label{fFdt}
\Ee
where the constants are independent of $\dt\in[0,1]$. Indeed, \eqref{fF1} follows from \eqref{fFdt}
applied to $\dt=0$ and $\dt=1/2$.

We now prove \eqref{fFdt} for a fixed $\dt\in[0,1]$. 
In the proof below we denote by $\Dt_{j,\mu}$ the set of discontinuity points of $\hdt$, that is
\[
\Dt_{j,\mu}=\big\{(\mu+\dt+i)2^{-j}\mid i=0,\tfrac12,1\big\}.
\]

\begin{lemma}\label{L_41}
Let $g\in L_1^{\rm loc}(\SR)$, $k\in\SN_0$ and $j\geq-1$. Then
\sline a) If $k\geq j$ then
\Be
\big|2^j\lan g{L_k\hdt} \big|\lesssim 2^{-(k-j)}\,\sum_{z\in\Dt_{j,\mu}}\fM_k(g)(z);
\label{41a}
\Ee
moreover, for every $A>0$,
\Be
\begin{split}
\big|2^j\lan g{L_k\hdt}\big|\,\bbone_{I_{j,\mu}}(x)\, & \lesssim\, 2^{-(k-j)}\,\fM^*_{j}(g)(x)\\
& \lesssim\, 2^{-(k-j)}\,2^{(k-j)A}\,\fM^{**}_{k,A}(g)(x).
\end{split}
\label{41A}
\Ee

\sline b) If $j>k$ then, for every $A>0$,
\Be
\begin{split}
\big|2^j\lan g {L_k\hdt}\big|\,\bbone_{I_{j,\mu}}(x) & \lesssim\, 2^{-(j-k)}\,\fM_{k}(g)(x)\\
& \lesssim\, 2^{-(j-k)}\,\fM^{**}_{k,A}(g)(x).
\end{split}
\label{41b}
\Ee
\end{lemma}
\begin{proof}
\sline a) If $k\geq j$ then the function $L_k\hdt=\beta_k*\hdt$ is supported in $\Dt_{j,\mu}+O(2^{-k})$ and has size
\[
\|\beta_k*\hdt\|_\infty\leq \|\beta_k\|_1=O(1).
\]
This immediately gives \eqref{41a}. Now, if $z\in\Dt_{j,\mu}$ and $x\in I_{j,\mu}$ we have
\Beas
\fM_kg(z)  =  \sup_{|h|\leq 2^{-k}}|g(z+h)| & \leq &  \sup_{|u|\leq 2^{-j+2}}|g(x+u)|= \fM^*_j(g)(x)\\
& \lesssim & 
2^{(k-j)A}\,\fM^{**}_{k,A}(g)(x),
\Eeas
which together with \eqref{41a} proves \eqref{41A}.

\sline b) If $k<j$ and $x\in I_{j,\mu}$, then the function $L_k\hdt=\beta_k*\hdt$ is supported in $x+O(2^{-k})$,
and we can bound its size by
\Be
|L_k\hdt(u)|\lesssim 2^{2(k-j)}.
\label{h2kj}
\Ee
This last assertion follows from the property $\int\hdt=0$, by writing
\Beas
|L_k\hdt(u)| & = & \Big|\int \big(\beta_k(u-y)-\beta_k(u-x)\Big)\,\hdt(y)\,dy\Big|\\
& = & \Big|\int\int_0^1 \beta'_k\big(u-(1-t)x-ty\big)\,dt\,(x-y)\,\hdt(y)\,dy\Big|\\
& \lesssim & 2^{2k-2j},
\Eeas
using in the last step that $|x-y|\leq 2^{-j+1}$ when $x\in I_{j,\mu}$ and $y\in \supp \hdt$.
Combining the above support and size estimates, one easily obtains \eqref{41b}.
\end{proof}

We continue the proof of Proposition \ref{prop1}.i. Let $f\in F^s_{p,q}$, and write it as $f=\sum_{k=0}L_kf_k$
with $f_k=\La_kf$ as in \eqref{LLak}. Note that, since $F^s_{p,q}\subset\sB$, we have
\[
\lan f\hdt = \sum_{k=0}^\infty\lan {f_k}{L_k\hdt}.
\]
Now, the estimates in Lemma \ref{L_41}, suitably applied to each $f_k$, can be grouped into
\Bea
\sum_{\mu}\big|2^j\lan f\hdt\big|\,\bbone_{I_{j,\mu}}(x) & \lesssim & 
\sum_{k\geq0} 2^{-|k-j|} 2^{A\,(k-j)_+}\,\fM^{**}_{k,A}(f_k)(x)\nonumber \\
\mbox{{\tiny ($\ell=k-j$)}} &		\lesssim & \sum\limits_{\ell \in \bbZ} a(\ell,A)\,\fM^{**}_{j+\ell,A} [f_{j+\ell}] (x)\,
=:G_j(x),\label{pointwise}
\Eea
where we set $f_m\equiv 0$ for $m<0$ and 
\begin{equation}\label{f6}
	a(\ell,A) = \left\{\begin{array}{rcl}
		2^{\ell}&:& \ell<0,\\
		2^{(A-1)\ell}&:& \ell\geq 0\,.
	\end{array}\right.
\end{equation}
At this point one takes $L_p(\ell_q)$ quasi-norms of the above expressions.
Letting $u:=\min\{p,q,1\}$, and using the $u$-triangle inequality we obtain
\begin{equation}\label{f9}
\begin{split}
&\Big\|\Big(\sum\limits_{j \ge -1}\Big|2^{js}G_j\Big|^q\Big)^{1/q}\Big\|_p\\
&\lesssim \Big(\sum\limits_{\ell \in \bbZ}
\Big[a(\ell,A)2^{-s\ell}\,\Big\|\Big(
\sum\limits_{j\in\SZ}|2^{(j+\ell)s}\fM^{**}_{j+\ell,A}[f_{j+\ell}](x)|^q\Big)^{1/q}\Big\|_p\Big]^u\Big)^{1/u}\\
&\lesssim \|f\|_{F^s_{p,q}}\;\Big(\sum\limits_{\ell\in\SZ}\big(a(\ell,A)2^{-s\ell}\big)^u\Big)^{1/u}\,\lc\|f\|_{F^s_{p,q}} ,
\end{split}
\end{equation}
where in the last line we use Peetre's maximal inequality \eqref{Peetremax} and  $A>\max\{1/p,1/q\}$,
 and in the last step we additionally need that $A-1 < s < 1$. This can always be achieved for an appropriate choice 
of $A$ because of our assumption $\max \{\frac 1p,\frac 1q\}-1<s<1.$

\

As before, part (ii) in Proposition \ref{prop1} will be a consequence of the more general estimate
\Be
\Big(\sum\limits_{j=-1}^{\infty} 2^{j(s-1/p)q}\Big(\sum\limits_{\mu\in \bbZ}| 2^j\lan f \hdt |^p\Big)^{q/p}\Big)^{1/q} \lc \,
\|f\|_{B^s_{p,q}},
\label{bBdt}
\Ee
for $\dt\in[0,1]$. 
Notice that
\Be
B_j:=2^{-j/p}\Big(\sum_{\mu}\big|2^j\lan f\hdt\big|^p\Big)^\frac1p=\Big\|\sum_\mu\big|2^j\lan f\hdt\big|\,\bbone_{I_{j,\mu}}\Big\|_{p}.
\label{Bj}
\Ee
We shall argue a bit differently to refine the pointwise estimate in \eqref{pointwise}.
Observe from Lemma \ref{L_41} that we can also write 
\Beas
\Big\|\sum_\mu\big|2^j\lan {f_k}\hdt\big|\,\bbone_{I_{j,\mu}}\Big\|_{p} & \lesssim &  2^{-|k-j|}\big\|\fM^*_{\min\{j,k\}}(f_k)\big\|_p\\
\mbox{{\tiny by \eqref{5.8}}} & \lesssim & 2^{-|k-j|}\,2^{\frac{(k-j)_+}p}\,\big\|\fM_{k}(f_k)\big\|_p \\
& = & a(k-j, \tfrac1p)\,\big\|\fM_{k}(f_k)\big\|_p ,
\Eeas
using 
in the last step the definition of $a(\ell,A)$ in \eqref{f6}.
So, letting as before $u:=\min\{p,q,1\}$, and using the $u$-triangle inequality we obtain
\Be
\begin{split}
\Big(\sum_{j\geq-1}(2^{js}&B_j)^q\Big)^\frac1q  \lesssim   
\\
 &\hskip-1cm\lesssim  \Big(\sum\limits_{\ell \in \bbZ}
\Big[a(\ell,\tfrac1p)\,2^{-s\ell}\,\Big(
\sum\limits_{j\in\SZ}|2^{(j+\ell)s}\big\|\fM_{j+\ell}[f_{j+\ell}]\big\|_{p}^q\Big)^{1/q}\Big]^u\Big)^{1/u}\\
&\hskip-1cm \lesssim  \|f\|_{B^s_{p,q}}\;\Big(\sum\limits_{\ell\in\SZ}\big(a(\ell,\tfrac1p)\,2^{-s\ell}\big)^u\Big)^{1/u}\,\lc\|f\|_{B^s_{p,q}}.
\end{split}
\label{f41}
\Ee
Observe that this time we apply the simpler estimate $\|\fM_{k}[f_{k}]\|_{p}\lesssim\|f_k\|_p$, see \eqref{fM1},
while the very last step requires $1/p-1 < s < 1$. 
\end{proof}

\begin{remark}\label{R_fF}
In view of the examples in \cite{SeeUl17}, the condition $s>1/q-1$ is necessary in Proposition \ref{prop1}.i, even for
the validity of the weaker inequality
\Be
	\|f\|_{F^{s, \rm{dyad}}_{p,q}}:=
	\Big\|\Big(\sum\limits_{j=-1}^{\infty} 2^{jsq}\Big|\sum\limits_{\mu\in \bbZ} 
	2^j\langle f,{h}_{j,\nu}\rangle
	\bbone_{I_{j,\mu}}\Big|^q \Big)^{1/q}\Big\|_p 
	\lesssim \|f\|_{F^s_{p,q}}\,.
	\label{fF}
	\Ee
Indeed, arguing as in \cite[\S6]{SeeUl17}, for each $N\geq2$ one constructs a (smooth) function $f=f_N\in F^{1/q-1}_{p,q}(\SR)$ 
and a finite set\footnote{
In the notation of \cite[\S6]{SeeUl17}, one should consider sets $A$ of \emph{consecutive} Haar frequencies, 
so that the associated ``density'' number in \cite[(43)]{SeeUl17} takes the value $Z=N$.} 
$E=E_N\subset\sH$ 
such that if $0<q\leq p<\infty$ then
\Be
\|f\|_{F^{1/q-1}_{p,q}(\SR)}\lesssim N^{1/q} \mand \|P_{E}(f)\|_{F^{1/q-1}_{p,q}(\SR)} \geq N^{1+\frac1q},
\label{fNR}
\Ee
where $P_E(f)$ denotes the projection onto $\Span\{h\}_{h\in E}$. Now observe that  
\[
\|f\|_{F^{s, \rm{dyad}}_{p,q}}\geq \|P_E(f)\|_{F^{s, \rm{dyad}}_{p,q}} \gtrsim \|P_E(f)\|_{F^{s}_{p,q}},
\] 
using in the last step the inequality in \eqref{f2} (which holds for $s=1/q-1$).  
Thus, we conclude from \eqref{fNR} that
\[
{\|f\|_{F^{1/q-1, \rm{dyad}}_{p,q}}}\gtrsim N\,{\|f\|_{F^{1/q-1}_{p,q}}}.
\]
\end{remark}

\

We turn to the converse inequalities in Theorems \ref{mainT-intro} and \ref{mainB-intro},
which will be proved in Proposition \ref{main1} below.
Before doing so we shall need some results concerning the inclusions $F^{s, \rm{dyad}}_{p,q}\hookrightarrow F^s_{p,q}$
(and likewise for $B$-spaces) for the usual Haar system. 
Results of this nature can be found in \cite[Proposition 2.6]{Tr10}, but
we give here direct proofs which are valid in a larger range of indices.
These may have an independent interest.

The first result is obtained by duality from Proposition \ref{prop1}, but
the range of indices is restricted to  $1<p,q<\infty$. Recall the definitions of the sequence spaces  $b^s_{p,q}$, $f^s_{p,q}$ in \eqref{b_sequence} and \eqref{f_sequence}.

\begin{proposition} \label{prop:duality}Let $1<p,q<\infty$.

(i) If $-1<s < \min\{1/p,1/q\}$ and $\beta=\{\beta_{j,\mu}\} \in f^s_{p,q}$ then 
\begin{equation}\label{decompo}
        f:=\sum\limits_{j=-1}^\infty \sum\limits_{\mu \in \bbZ} \beta_{j,\mu}h_{j,\mu}
\end{equation}
converges in $F^s_{p,q}(\bbR)$ and 
  \begin{equation}\label{duality}
	\|f\|_{F^s_{p,q}} \lesssim  \|\beta\|_{f^s_{p,q}}
  \end{equation}

  (ii) If $-1<s<1/p$ and $\beta=\{\beta_{j,\mu}\} \in b^s_{p,q}$ then \eqref{decompo} 
converges in $B^s_{p,q}(\bbR)$ and
  \begin{equation} \label{duality{besov} } 
  \|f\|_{B^s_{p,q}} \lc \|\beta\|_{b^s_{p,q}}.
  \end{equation}
\end{proposition}
\begin{proof} Consider the following duality pairing for sequences
\begin{equation}\nonumber
  \begin{split}	
  &\biginn{\alpha}{\beta} :=
	\sum\limits_{j}2^{-j}\sum\limits_{\mu} 
	\alpha_{j,\mu}\,\beta_{j,\mu} 
\\
	&= \sum\limits_{j} \int_{-\infty}^{\infty}\Big(\sum\limits_{\mu}\alpha_{j,\mu} \bbone_{I_{j,\mu}}(x)\Big)\Big(
	\sum\limits_{\mu'}\beta_{j,\mu'}\bbone_{j,\mu'}(x)\Big)\,dx
	\leq \|\alpha \|_{f^s_{p,q}} \|\beta \|_{f^{-s}_{p',q'}}\,.
  \end{split}	
\end{equation}
We define the so-called \emph{analysis} (or sampling) \emph{operator} $\sA$ by
\begin{equation}\label{OpA}
	\begin{split}
		&\sA: \, F^s_{p,q}(\bbR) \longrightarrow f^s_{p,q}\\
		  &\quad \quad f\longmapsto \sA f(j,\mu) = 2^j\langle f,h_{j,\mu }\rangle.
	\end{split}
\end{equation}
Its dual operator $\sA'$ (the  \emph{synthesis operator}) is given by
\begin{equation}\label{OpA'}
\begin{split}
	&\sA'\beta(x) =
	\sum\limits_{j,\mu} \beta _{j,\mu} h_{j,\mu}(x) \,.
\end{split}\end{equation}
Indeed, if $f\in F^s_{p,q}$
and  $\beta  = \{\beta_{j,\mu} \}$ 
is a finite sequence, then 
\begin{equation}\nonumber
  \begin{split}
	\langle \sA f, \beta\rangle &= \sum\limits_{j} 2^{-j} \sum\limits_{\mu} 2^j\langle f,h_{j,\mu }\rangle\cdot \beta_{j,\mu } \\
	&= \int f(x)  \sum\limits_{j}\sum\limits_{\mu}  \beta_{j,\mu}h_{j,\mu}(x)\,dx= \int f(x)\,\sA'\beta(x)\,dx.
  \end{split}
\end{equation}
By Proposition \ref{prop1}, the operator
$\sA: F^s_{p,q}(\bbR)\to f^s_{p,q}$ is bounded when $\max\{1/p-1,1/q-1\}<s<1$. 
Hence, if we assume that $1\leq p,q<\infty$, then $\sA'$ will be bounded from $f^{-s}_{p',q'}$ to $F^{-s}_{p',q'}(\bbR)$, where 
$$
	-1<-s<\min\{1/p',1/q'\}\,.
$$
In other words, if $-1<s<\min\{1/p,1/q\}$ and $1<p,q\leq \infty$ then 
$$
	\Big\|\sum\limits_{j,\mu} \beta_{j,\mu }\h_{j,\mu }\Big\|_{F^s_{p,q}} \lesssim \|\beta\|_{f^s_{p,q}}\,.
$$
The proof of part (ii) is completely analogous.
\end{proof}

In the following proposition we extend the previous estimate to the quasi-Banach range of parameters.
The proof, which is now more involved, uses a non-trivial interpolation argument from \cite[Proposition 2.6]{Tr10}.  
The range of indices we obtain is larger than in \cite{Tr10}; see Figure \ref{fig_prop4} below.

\begin{figure}[h]
 \centering
\subfigure
{
\begin{tikzpicture}[scale=1.45]

\node [right] at (1.5,-1) {{\footnotesize $F^s_{p,q}$-spaces}};

\draw[->] (-0.1,0.0) -- (3.1,0.0) node[below] {$1/p$};
\draw[->] (0.0,-0.0) -- (0.0,1.6) node[above] {$s$};
\draw (0.0,-1.1) -- (0.0,0)  ;

\draw (1.0,0.03) -- (1.0,-0.03) node [below] {$1$};
\draw (2.0,0.03) -- (2.0,-0.03) node [below] {$2$};
\draw (0.03,1.0) -- (-0.03,1.00);
\node [left] at (0,0.9) {$1$};
\draw (0.03,-1.0) -- (-0.03,-1.00) node [left] {$-1$};

\node [left] at (0,-0.75) {{\tiny $\tfrac1q-2$}};

\node [left] at (.75,.75) {{\tiny $s=\tfrac1p$}};
\node [right] at (2.25, 0.25) {{\tiny $s=\tfrac1p\!-\!2$}};

\draw[dotted] (0,-1) -- (2.5,1.5);
\draw[dotted] (0,1) -- (1,1)--(1,-1);
\draw[dotted]  (1.25,1.25)-- (0,1.25) node[left] {{\tiny $1/q$}};
\draw[dotted] (0,-1)--(1,-1)--(1.25,-0.75);

\draw[dashed, thick]  (0,0)--(1.25,1.25)--(3.25,1.25)-- (1.25,-0.75)--(0,-.75)--cycle;

\draw[white, fill=red!70, opacity=0.4] (0,0)--(1.25,1.25)--(3.25,1.25)-- (1.25,-0.75)--(0,-.75)--cycle;

\end{tikzpicture}
}
\subfigure
{
\begin{tikzpicture}[scale=1.45]

\node [right] at (1.5,-1) {{\footnotesize $B^s_{p,q}$-spaces}};

\draw[->] (-0.1,0.0) -- (3.1,0.0) node[below] {$1/p$};
\draw[->] (0.0,-0.0) -- (0.0,1.6) node[above] {$s$};
\draw (0.0,-1.1) -- (0.0,0)  ;

\draw (1.0,0.03) -- (1.0,-0.03) node [below] {$1$};
\draw (2.0,0.03) -- (2.0,-0.03) node [below] {$2$};
\draw (0.03,1.0) -- (-0.03,1.00);
\node [left] at (0,0.9) {$1$};
\draw (0.03,-1.0) -- (-0.03,-1.00) node [left] {$-1$};

\draw[dotted] (0,-1) -- (2.5,1.5);
\draw[dotted] (0,1) -- (1,1)--(1,-1);

\draw[dashed, thick]  (2,2)--(0,0)-- (0.0,-1.0)--(1,-1)--(3,1);

\draw[white, fill=red!70, opacity=0.4] (2,2)--(0,0)-- (0.0,-1.0)--(1,-1)--(3,1)--cycle;

\end{tikzpicture}

}
\caption{Parameter domain for the validity of Proposition \ref{prop4}, for $F^s_{p,q}(\SR)$ (left figure) 
and $B^s_{p,q}(\SR)$ (right figure). 
}\label{fig_prop4}
\end{figure}
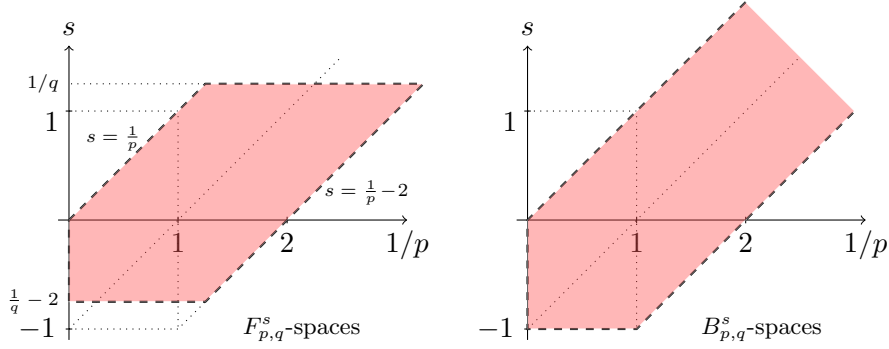

\begin{proposition}\label{prop4}Let $0<p, q\leq \infty$.
\sline (i) If $p<\infty$ and $\max\{1/p,1/q,1\}-2<s<\min\{1/p,1/q\}$,
then for every $\beta \in f^s_{p,q}$ the series in
\eqref{decompo} converges to a distribution $f$ in $F^s_{p,q}(\bbR)$, and it holds    
\begin{equation}\label{f2}
	\|f\|_{F^s_{p,q}} \lesssim \|\beta\|_{f^s_{p,q} }.
\end{equation}
Moreover if $q<\infty$ the series 
in  \eqref{decompo} converges unconditionally in  $F^s_{p,q}$, and otherwise in  $F^{s-\varepsilon}_{p,\infty}(\bbR)$, for all $\varepsilon>0$.

\sline (ii) If $\max\{1/p,1\}-2<s<1/p$,
then for every $\beta \in b^s_{p,q}$ the series in
\eqref{decompo} converges to a distribution $f$ in $B^s_{p,q}(\bbR)$, and it holds
\begin{equation}\label{f2-Besov} 
\|f\|_{B^s_{p,q}}\lc \|\beta\|_{b^s_{p,q}}.
\end{equation} 
Moreover, if $q <\infty$ the series in \eqref{decompo} converges unconditionally in the
norm of $B^s_{p,q}$, and otherwise in the quasi-norm of $B^{s-\varepsilon}_{p,\infty}(\bbR)$, for all $\varepsilon>0$.
\end{proposition}

\begin{proof} It suffices to prove the results when $f=\sum_{j\geq-1}\sum_{\mu \in \bbZ} \beta_{j,\mu}h_{j,\mu}$,
with $(\beta_{j,\mu})$ a finite sequence of scalars. The other assertions will then follow by completeness of the spaces.

\sline {\em Step 1. } Let $L_k$, $k\geq0$, be the local convolution operators from Section \ref{sect:MF_FS}. 
If $j\geq-1$ is fixed, then we have
\begin{equation}\label{f11}
		\Big|L_k \Big(\sum\limits_{\mu\in \bbZ}\beta_{j,\mu}h_{j,\mu}\Big)(x)\Big| \leq \sum\limits_{\mu\in \bbZ}
		|\beta_{j,\mu}|\cdot |(L_k h_{j,\mu})(x)|\,=:\,G_{j,k}(x).
\end{equation}
Arguing as in Lemma \ref{L_41}, one sees that, in the case $\ell=j-k>0$, then
\Be
\supp L_k h_{j,\mu}\subset \mu2^{-j}+O(2^{-k})\mand \big|(L_k h_{j,\mu})(x)\big| \lesssim 2^{-2\ell},
\label{jk1}
\Ee
while in the case  $\ell=j-k \leq 0$ we have
\Be
\supp L_k h_{j,\mu}\subset \mu2^{-j}+O(2^{-j})\mand \big|(L_k h_{j,\mu})(x)\big| \lesssim 1.
\label{jk2}
\Ee
The following lemma is a variation of a result by Kyriazis; see  \cite[Lem. 7.1]{Ky03}.
For $r>0$ we denote 
\[
M_r(g)(x)=\big[M(|g|^r)(x)\big]^{1/r},
\]
where $M$ is the usual Hardy-Littlewood maximal operator.

\begin{lemma}\label{L_kyr}
Let $0<r\leq 1$. Then, 
it holds
\Be
G_{j,k}(x)\,\lesssim\, a(j-k)\,
M_r\Big(\sum_\mu|\beta_{j,\mu}|\,\bbone_{I_{j,\mu}}\Big)(x).
\label{f10}
\Ee
where $a(\ell)=2^{-2\ell}2^{{\ell}/r}$ if $\ell>0$, and $a(\ell)=1$ if $\ell\leq0$.
\end{lemma}
\begin{proof}
Assume that $\ell=j-k>0$. If $x\in I_{j,\nu}$, for some fixed $\nu\in\SZ$, then the size and support estimates in \eqref{jk1}
give
\[
G_{j,k}(x)^r\lesssim\,\sum_{\mu:|\mu-\nu|\leq 2^{j-k}}|\beta_{j,\mu}|^r\,2^{-2\ell r}.
\]
On the other hand, if $x\in I_{j,\nu}$,
\Beas
M\Big(\sum_\mu|\beta_{j,\mu}|^r\,\bbone_{I_{j,\mu}}\Big)(x) & \gtrsim & 2^k\int_{|y-2^{-j}\nu|\leq 2\cdot 2^{-k}}
\Big(\sum_\mu|\beta_{j,\mu}|^r\,\bbone_{I_{j,\mu}}\Big)(y)\,dy\\
& \geq & 2^{k-j}\,\sum_{\mu:|\mu-\nu|\leq 2^{j-k}}|\beta_{j,\mu}|^r\,.
\Eeas
These two estimates clearly give \eqref{f10} in the case $\ell>0$. The case $\ell\leq0$ is proved similarly using \eqref{jk2}.
\end{proof}


We continue with the proof of Proposition \ref{prop4}.i. 
Below we shall agree that $\beta_{m, \mu}=0$ and $G_{m,k}\equiv0$ whenever $m<-1$.
Then
letting $u = \min\{p,q,1\}$, we can apply the $u$-triangle inequality
and the above results to obtain
\Beas
\|f\|_{F^s_{p,q}} & = & \Big\|\Big(\sum_{k\geq0}|2^{ks}L_kf|^q\Big)^{1/q}\Big\|_p\, \leq \,
 \Big\|\Big(\sum_{k\geq0}(2^{ks}\sum_{\ell\in\SZ}G_{k+\ell,k})^q\Big)^{1/q}\Big\|_p\\
& \hskip-2cm \lesssim & \hskip-1cm \Big[\sum\limits_{\ell\in\SZ}\Big(a(\ell)\,2^{-\ell s}\Big)^u
\Big\|\Big(\sum\limits_{k\geq0 }\Big[2^{(k+\ell)s}\,M_r\Big(\sum\limits_{\mu\in \bbZ}|\beta_{k+\ell,\mu}|\bbone_{I_{k+\ell,\mu}}\Big)
\Big]^{q}\Big)^{\frac 1q}\Big\|^u_p\Big]^{\frac 1u}\\
		& \hskip-2cm \lesssim & \hskip-1cm\Big\|\Big(\sum\limits_{m\in\SZ} \Big(2^{ms}\sum\limits_{\mu\in \bbZ} 
	|\beta_{m,\mu}|\bbone_{I_{m,\mu}}\Big)^q \Big)^{\frac 1q}\Big\|_p\,,
\Eeas
where the last line is justified by the Fefferman-Stein inequality \eqref{HLmaxineq}, provided
$r<\min\{p,q,1\}$, and the finite summation in $\ell\in\SZ$  holds whenever $1/r-2<s<0$\,. Such an $r$ can always be 
chosen under the assumption 
\[
\max\{1/p,1/q,1\}-2<s<0
\]
(which in particular implies $p,q>1/2$).
We shall see in Step 3 below how to enlarge this range to cover as well the cases $s\geq0$.

{\em Step 2. } We now prove \eqref{f2-Besov}. The same notation as above gives
\[
\|f\|_{B^s_{p,q}} = \Big(\sum_{k\geq0}\big(2^{ks}\|L_kf\|_p\big)^q\Big)^{1/q}
\leq \Big(\sum_{k\geq0}(2^{ks}\big\|\sum_{\ell\in\SZ}G_{k+\ell,k}\big\|_p)^q\Big)^{1/q}.
\]
At this point we distinguish two cases, $\ell> 0$ and $\ell\leq 0$. In the first case we use literally the same arguments as above;
since for the $\ell_q(L_p)$ quasi-norm we just use the \emph{scalar} Hardy-Littlewood maximal inequality 
we only need to impose $r<\min\{1,p\}$, together with $s>1/r-2$. Such an $r$ can always be chosen under the assumption
\[
\max\{1/p,1\}-2<s.
\]
To control the sum over $\ell\leq0$ we must replace the crude bound in \eqref{jk2} by the sharper estimate
\Be
\label{jk3}
\supp L_k h_{j,\mu}\subset \Dt_{j,k}+O(2^{-k})\mand \big|(L_k h_{j,\mu})(x)\big| \lesssim 1,
\Ee
where $\Dt_{j,\mu}$ are the discontinuity points of $h_{j,\mu}$; see the proof of Lemma \ref{L_41}.a. 
So, if $\ell=j-k\leq0$ we have
\[
\big\|G_{k+\ell,k}\big\|_p\lesssim 2^{-k/p}\Big(\sum\limits_{\mu \in \bbZ} |\beta_{k+\ell,\mu}|^p\Big)^{1/p}.
\]
Using as before the $u$-triangle inequality, this yields
\[
\begin{split}
\Big(\sum_{k\geq0} &\Big (2^{ks}\big\|\sum_{\ell\leq0}G_{k+\ell,k}\big\|_p\Big)^q\Big)^{1/q}\\ 
 & \lesssim\; \Big[\sum_{\ell\leq0}2^{u(\frac1p-s)\ell}\Big(\sum_{k\geq0}\Big\{ 2^{(k+\ell)(s-\frac1p)}
\Big(\sum\limits_{\mu \in \bbZ} |\beta_{k+\ell,\mu}|^p\Big)^{1/p}\Big\}^q\Big)^{\frac uq}\Big]^\frac1u
\\ 
 & \lesssim\; \Big(\sum_{\ell\leq0}2^{u(\frac1p-s)\ell}\Big)^\frac1u\,\Big(\sum_{m\in\SZ}\Big\{2^{m(s-\frac1p)}
\Big(\sum\limits_{\mu \in \bbZ} |\beta_{m,\mu}|^p\Big)^{1/p}\Big\}^q\Big)^{\frac 1q},
\end{split}
\]
where the sum in $\ell\leq0$ is a finite constant due to the assumption $s<1/p$.
This completes the proof of part (ii) in Proposition \ref{prop4}.

{\em Step 3.} In the Triebel-Lizorkin case, the direct argument in Step 1 only allows for $s<0$
(and $p,q>1/2$),
which is the desired region only when $q=\infty$ or $p\to\infty$. 
By Step 2, the range of parameters can be extended to $s<1/p$ when $p=q$.
Then, a complex interpolation argument in the three indices $(s,1/p,1/q)$, 
as proposed by Triebel in \cite[Prop.\ 2.6]{Tr10},
gives the validity of the result for all $\max\{1/p,1/q,1\}-2<s < \min\{1/p,1/q\}$;
see Figure \ref{fig_spq}.
\end{proof}

\newcommand\zeugs{
		\coordinate (O) at (0,0,0);
		\coordinate (P) at (2,0,0);
		\coordinate (Q) at (0,2,0);
		\coordinate (PQ) at (2,2,0);
		\coordinate (A) at (4,4,3);		
		\coordinate (B) at (4,4,3-1.5);
		\coordinate (Axy) at (4,4,0);

		\draw[dashed,fill=red,fill opacity=.1] (O) -- (P) -- (PQ) --(Q) --cycle;
		\draw[fill=red,opacity=.5] (O) -- (A) -- (B) -- (PQ) --cycle;
		
		\draw[-latex](O)--(4,0,0)node[pos=1.1]{$\frac1p$};
		\draw[-latex](O)--(0,4,0)node[pos=1.1]{$\frac1q$};
		\draw[-latex](O)--(0,0,1.5)node[pos=1.1]{$s$};
		\draw[fill](P)circle(.75pt)node[above]{2};
		\draw[fill](Q)circle(.75pt)node[above]{2};

		\draw[thick,dotted] (PQ) --++(0,0,1.5);
		\draw[thick,dotted] (PQ) -- (Axy) -- (B);
		\draw[thick,dotted] (4,0,0)--(Axy)--(0,4,0);
}
\begin{figure}[h]\center
	\tdplotsetmaincoords{70}{140}		
	\subfigure{
	\begin{tikzpicture}[scale=0.9,tdplot_main_coords]
		\zeugs
	\end{tikzpicture}}%
	\subfigure{
	\begin{tikzpicture}[scale=0.9,tdplot_main_coords]
		\zeugs
		\draw[fill,blue,fill opacity=.2](O)--(A)--(Q)--cycle;
		\draw[fill,blue,fill opacity=.2](O)--(A)--(P)--cycle;
		\draw[fill,Purple,fill opacity=.2](A)--(B)--(Q)--cycle;
		\draw[fill,Purple,fill opacity=.2](A)--(B)--(P)--cycle;
		\draw[fill,Green,fill opacity=.2](B)--(PQ)--(Q)--cycle;
		\draw[fill,Green,fill opacity=.2](B)--(PQ)--(P)--cycle;
	\end{tikzpicture}}

	\caption{Parameter domain  for  $F$-spaces in Steps 1 and 2 (left figure), and after the interpolation argument in Step 3 (right figure) of the proof of Proposition \ref{prop4}.}
	\label{fig_spq}
\end{figure}
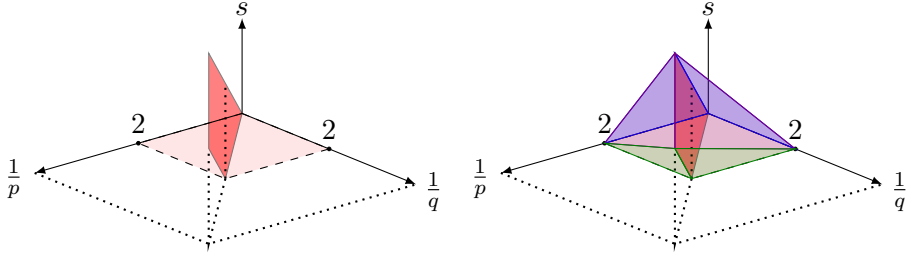
\begin{remark}\label{R_dt}
We remark that the decomposition of a distribution $f\in\cS'$ as an infinite series
\Be
f=\sum_{j\geq-1}\sum_{\mu \in \bbZ} \beta_{j,\mu}h_{j,\mu}
\label{fbj}
\Ee
may not necessarily be unique. For instance, the Dirac delta satisfies
\[
\dt=\bbone_{[0,1)}+\sum_{j=0}^\infty 2^j h_{j,0}=\bbone_{[-1,0)}-\sum_{j=0}^\infty 2^j h_{j,-1} \quad \mbox{in $S'(\SR)$.}
\]
In this example, the coefficient sequences belong to $b^s_{p,q}$ if 
$s<\frac1p-1$ (or $s=\frac1p-1$ and $q=\infty$),
and the same happens for the property $\dt\in B^s_{p,q}(\SR)$.
For such cases of non-uniqueness, Proposition \ref{prop4} should be interpreted as 
\[
\|f\|_{B^s_{p,q}}\lesssim \inf\Big\{\|(\beta_{j,\mu})\|_{b^s_{p,q}}\mid \eqref{fbj} \mbox{ holds }\Big\},
\]
and likewise for the $F^s_{p,q}$-quasinorms.
\end{remark}

The next result shows that uniqueness holds when $s>1/p-1$.

\begin{corollary}\label{cor5.5} Let $0<p,q\leq\infty$ and $s\in\SR$.

 (i) If $p<\infty$ and $\max\{1/p-1,1/q-2\} < s < \min\{1/p,1/q\}$
then for all $f \in \sB$ it holds
\begin{equation}\label{cor5.5_F}
    \|f\|_{F^s_{p,q}} \lesssim \Big\|\Big(\sum\limits_{j=-1}^{\infty}\Big|2^{js}\sum\limits_{\mu\in \bbZ} 
		2^j\langle f, h_{j,\mu}\rangle\bbone_{I_{j,\mu}}\Big|^q\Big)^{1/q}\Big\|_p\,.
\end{equation}
(ii) If $~1/p-1 < s < 1/p$
then for all $f \in \sB$ it holds
\begin{equation}\label{cor5.5_B}
    \|f\|_{B^s_{p,q}} \lesssim \Big(\sum\limits_{j=-1}^{\infty} 2^{j(s-1/p)q}
		\Big(\sum\limits_{\mu\in \bbZ}|2^j\langle f,h_{j,\mu}\rangle|^p\Big)^{q/p}\Big)^{1/q}.
\end{equation}
\end{corollary}

\begin{proof} 
Let $f\in\sB$ be such that the right hand side of \eqref{cor5.5_F} is finite.
By Proposition \ref{prop4} this implies the convergence of the series 
$$
    g:=\sum\limits_{j=-1}^{\infty} \sum\limits_{\mu \in \bbZ} 2^j\langle f,h_{j,\mu} \rangle h_{j,\mu}\,,
$$
to some distribution $g\in F^s_{p,q}\hookrightarrow\sB$. 
Due to the range of parameters, 
and the convergence in $F^{s-\varepsilon}_{p,q}(\bbR)$ for $\varepsilon$ small enough, 
we also have convergence in $\sB$.
We deduce that $\langle g, h\rangle = \langle f,h\rangle$ for all $h \in \sH$,
and therefore, by Proposition \ref{prop:test-funct}, that $f=g$. Finally, Proposition \ref{prop4} gives \eqref{cor5.5_F}. 
The proof for \eqref{cor5.5_B} works analogously. 
\end{proof}


We finally turn to the remaining implications in Theorems \ref{mainT-intro} and \ref{mainB-intro},
which are also valid in a larger range.

\begin{proposition}\label{main1} Let $0<p, q\leq \infty$ and $s\in\SR$ be such that
\[
\frac1p-1<s<1+\frac1p.
\]

(i) If $p<\infty$ and additionally 
$1/q-2<s<2+1/q$, then for all $f\in \sB$ 
\begin{equation}\label{f15}
	\|f\|_{F^s_{p,q}} \lesssim \|\fc(f)\|_{f^s_{p,q}}
\end{equation}

(ii) For all $f\in \sB$ it holds
\begin{equation}\label{f15-Besov}
\|f\|_{B^s_{p,q}}\lc \|\fc(f)\|_{b^s_{p,q}}
\end{equation}
\end{proposition}

\begin{proof}  Note that, for all $f\in \mathcal{S}'(\bbR)$,  it holds 
\begin{equation}\label{f3}
	\|f\|_{F^s_{p,q}} \approx \|f\|_{F^{s-1}_{p,q}} + \|f'\|_{F^{s-1}_{p,q}}\,,
\end{equation}
see e.g. \cite[2.3.8]{Tr83}\,. We shall bound each of
the summands in \eqref{f3} by 
the right hand side of \eqref{f15}.

Clearly $\|f\|_{F^{s-1}_{p,q}} \lesssim \|f\|_{B^{r}_{p,p}}$ for any $s-1 < r$. We distinguish two cases. 
In case $s>1/p$ we choose $r:=1/p-\varepsilon$, for a sufficiently small $\varepsilon>0$ so that
\Be
1/p-1<r<1/p \mand  s-1 < r < s;
\label{sr}
\Ee
this is possible by the assumption $s<1+1/p$. 
In case $s \leq 1/p$ we put $r := s-\varepsilon$, for some $\varepsilon >0$
so that \eqref{sr} also holds (this time using the assumption $s>1/p-1$). 
Hence, from
the embeddings $B^r_{pp}\hookrightarrow F^{s-1}_{pq}$ and $f^s_{pq}\hookrightarrow b^r_{pp}$, 
together with Corollary \ref{cor5.5}.ii, we obtain
\Be
\label{f14}
\|f\|_{F^{s-1}_{p,q}}\lesssim \|f\|_{B^{r}_{p,p}}\lesssim \|\{2^j\langle f,h_{j,\mu}\rangle\}_{j,\mu}\|_{b^r_{pp}}\lesssim
\|\{\fc_{j,\mu} (f)\}_{j,\mu}\|_{f^s_{pq}}.
\Ee

We now take care of  the second term in \eqref{f3}. Here we quote the analog of Corollary 
\ref{cor5.5}.i for the Chui-Wang system $\{\psi_{j,\mu}\}$, which can be obtained 
from \cite[Proposition 5.4]{DeUl20} and Remark \ref{R_DU} above.
Letting $r=s-1$, this gives 
\begin{equation}\label{f4}
	\|f'\|_{F^r_{p,q}} \lesssim \Big\|\Big(\sum\limits_{j\geq-1} 2^{jrq}\Big|\sum\limits_{\mu\in \bbZ} 
	2^j\langle f',\psi_{j,\mu}\rangle\bbone_{I_{j,\mu}}(x)\Big|^q \Big)^{1/q}\Big\|_p\,,
\end{equation}
provided
\[
\max\{1/p-1,1/q-2\}-1 < r < 1+\min\{1/p,1/q\},
\]
which holds when $1/p-1<s<1/p+1$ and $1/q-2<s<2+1/q$.


Now, recall that
\Be
\label{psiN}
\psi_{j,\mu}(x)=\sum_{k\in\SZ}b_k \cN_2(2^{j+1}x-(2\mu+k)\big),
\Ee
for a sequence of coefficients $b_k$ supported in $\{0,\ldots,4\}$; see 
\eqref{bk-coefficients}. 
In addition, we know from Lemma \ref{L_Nder} that
\Be
\label{haarvshat}
\langle f',\mathcal{N}_2(2^{j+1}\cdot -\nu) \rangle =-2^{j+1}\langle f, \widetilde{h}_{j,\nu} \rangle, \quad j\geq0,\;\nu\in\SZ
\Ee
(and a similar expression for $j=-1$).
So, combining \eqref{psiN} and \eqref{haarvshat}, we obtain an estimate for $|2^j\lan{f'}{\psi_{j,\mu}}|$ in terms
of the coefficients $\fc_{j,\mu+\ell}(f)$, $\ell\in\{0,1,2\}$, which inserted into \eqref{f4} gives
\Beas
	\|f'\|_{F^r_{p,q}}  & \lesssim & \sum_{\ell=0}^2
	\Big\|\Big(\sum\limits_{j\geq-1} 2^{jrq}\Big|\sum\limits_{\mu\in \bbZ} 
	\fc_{j,\mu+\ell}(f)\bbone_{I_{j,\mu}}(x)\Big|^q \Big)^{1/q}\Big\|_p\,
	\\
	& \lesssim &  \|(\fc_{j,\mu} (f))_{j,\mu}\|_{f^s_{pq}},
\Eeas
using in the last step a maximal function estimate as in Lemma \ref{L_kyr}.
This, together with \eqref{f14} concludes the proof of part (i).

The result for $B^s_{p,q}$ in part (ii) goes similarly, using instead Corollary \ref{cor5.5}.ii, 
and 
 the corresponding version for the Chui-Wang system $\{\psi_{j,\mu}\}$ which can be obtained 
from \cite[Proposition 5.3]{DeUl20}.
\end{proof}

\begin{remark}\label{R_sche}
As pointed out by one of the referees, in the specific case of Besov spaces $B^s_{p,q}(\SR)$ with $1/p\leq s<1$ (and $1<p\leq\infty$), one could give a more direct proof of the inequality
\[
\|f\|_{B^s_{p,q}}\lesssim \|\fc(f)\|_{b^s_{p,q}},\quad f\in L^1_{\rm loc}(\SR),
\]
by estimating the modulus of continuity $\om(f,2^{-j})_{L^p(\SR)}$ in terms of the errors of best linear approximation by piecewise constants over the collections of intervals $\{I_{j,\mu}\}_{\mu\in\SZ}$ and $\{I_{j,\mu}+2^{-j-1}\}_{\mu\in\SZ}$.  When $s>1$ (or $s=1$ and $q<\infty$) this argument also shows that \[
\|\fc(f)\|_{b^s_{p,q}}<\infty \implies 
\om(f,2^{-j})_{L^p(\SR)}=o(2^{-j}),\]
which in turn implies that  $f$ is constant (see e.g. \cite[Ch 2, Prop 7.1]{dVL}).  This type of reasoning is reminiscent of some works that appeared in the spline community in the 70s, see e.g. \cite[Theorem 3]{Sche74} or the references quoted in \cite[\S12.2]{dVL}. 
\end{remark}

We are finally ready to give the
\begin{proof}[Proof of Theorems  \ref{mainT-intro} and \ref{mainB-intro}]
Just combine Proposition \ref{prop1} and Proposition \ref{main1}. Note that the smallest range of
parameters corresponds to that in Proposition \ref{prop1}.
\end{proof}

\section{$W^1_p$ and $BV$: Proof of Theorem \ref{W1-char-intro} }
\label{sec:W1p}
The proof has three steps. We use the classical norm definition for $W^1_p(\bbR)$, when $1\leq p\leq \infty$, namely 
$$
	\|f\|_{W^1_p} := \|f\|_p +  \|f'\|_p\,.
$$



\subsection*{\it Step 1.} We show that, for $1\leq p\leq \infty$, it holds
\Be\label{BoundingW11Haarext}
\sup\limits_{j\geq -1} 2^{j(1-1/p)}\Big(\sum\limits_{\nu \in \bbZ} 
	|2^j\langle f,\widetilde{h}_{j,\nu}\rangle|^p\Big)^{1/p}\lesssim \|f\|_{W^1_p},
\quad f\in W^1_p(\SR),
\Ee
moreover
\Be\label{BoundingBV1Haarext}
\sup\limits_{j\geq -1} \sum\limits_{\nu \in \bbZ} 
	|2^j\langle f,\widetilde{h}_{j,\nu}\rangle|\lesssim \|f\|_{BV},
\quad f\in BV(\SR),
\Ee

In view of Lemma \ref{L_Nder}, for every $j\geq 0$ we have  
\begin{equation}
 \big|2^{j+1}\langle f,\widetilde{h}_{j,\nu} \rangle\big| \leq  \int |f'(x)\cN_{2;j+1,\nu}(x)|\,dx
\lesssim 2^{-j/p'}\Big[\int_{\tilde{I}_{j,\nu}}|f'|^p\,dx\Big]^{1/p},
\end{equation}
where $\tilde{I}_{j,\nu}=\supp\cN_{2;j+1,\nu}=[\nu/2^{j+1},(\nu+2)/2^{j+1}]$.
Hence,
$$
	\sup\limits_{j\geq 0} 2^{j(1-1/p)}\Big(\sum\limits_{\nu \in \bbZ} |2^j\langle f,\widetilde{h}_{j,\nu}\rangle|^p\Big)^{1/p}
	\lesssim \|f'\|_p\,.
$$
Likewise, if  $j = -1$ we simply have 
$$
	|\langle f, \widetilde{h}_{-1,\nu}\rangle| = \Big|\int_{{I}_{0,\nu}}f(x)dx\Big| \leq \Big(\int_{{I}_{0,\nu}}|f(x)|^p dx\Big)^{1/p}
$$
and hence 
$$
	\Big(\sum\limits_{\nu \in \bbZ} |\langle f, \widetilde{h}_{-1,\nu}\rangle|^p\Big)^{1/p} \lesssim \|f\|_p\,.
$$
We have thus established \eqref{BoundingW11Haarext}. To handle \eqref{BoundingBV1Haarext} we work with an approximation of the identity, $\{\Phi_\ell\}$ where $\Phi_\ell =2^\ell\Phi(2^\ell\cdot)$ with $\Phi\in C^\infty_c$ and $\int\Phi=1$. Let $f\in BV$ (which implies $f\in L_\infty$). Then $\Phi_\ell*f\in W^1_1$ with
$\|\Phi_\ell*f\|_{W^1_1} \lc \|f\|_{BV}$ and $\Phi_\ell*f(x)\to f(x) $ almost everywhere. By dominated convergence $\inn{\Phi_\ell*f}{\widetilde h_{j,\nu}}\to \inn{f}{\widetilde h_{j,\nu}}$ and by a further application of Fatou's lemma and \eqref{BoundingW11Haarext}
\begin{align*}
    \sum\limits_{\nu \in \bbZ} 
	|2^j\langle f,\widetilde{h}_{j,\nu}\rangle|
	\le \liminf_{\ell\to\infty} 
	\sum\limits_{\nu \in \bbZ} 
	|2^j\langle f*\Phi_\ell ,\widetilde{h}_{j,\nu}\rangle|
	\lc
	\liminf_{\ell\to\infty} \|\Phi_\ell*f\|_{W^1_1} 
	\lesssim \|f\|_{BV},
\end{align*}
where the implicit constants are independent of $j$. This yields \eqref{BoundingBV1Haarext}.


{\em Step 2.} We show that, for $1\le p\leq \infty$, we have
$$
	\|f\|_p 
	\,\lesssim\, \sup\limits_{j\geq -1} 2^{j(1-1/p)}\Big(\sum\limits_{\nu \in \bbZ} 
	|2^j\langle f,\widetilde{h}_{j,\nu}\rangle|^p\Big)^{1/p}\,=:\,A(p)=A.
$$
Recall from Lemma \ref{L_sB}.b that $\SE_Nf\to f$ in $\cS'$.
Assuming (as we may) that  $f\in\sB$ has finite right hand side it suffices to show that $\bbE_N f$ converges in $L_p$ and 
 $
\sup_{N\geq0}\big\|\SE_Nf\big\|_p\lesssim A.$
To see this, one expands \[\SE_Nf=\SE_0f+\sum_{0\leq j<N}\sum_{\mu\in\SZ} 2^j \lan f{h_{j,\mu}}h_{j,\mu},\]
and notes that 
\[
\big\|\sum_{\mu\in\SZ} 2^j \lan f{h_{j,\mu}}h_{j,\mu}\big\|_{p}=2^{-j/p}\Big(\sum_{\mu\in\SZ} |2^j \lan f{h_{j,\mu}}|^p\Big)^\frac1p
\leq 2^{-j}\,A,
\]
hence $\|\bbE_{N_1}f-\bbE_{N_2} f\|_p\lc 2^{-N_1}\, A$, for all $N_2> N_1$, and thus
\[\|f\|_p\le \sup_N\|\bbE_N f\|_p\le A.\]


{\em Step 3.} 
We finally show, for $1< p\leq \infty$, that 
$$
	\|f'\|_p 
	\,\lesssim\, \sup\limits_{j\geq -1} 2^{j(1-1/p)}\Big(\sum\limits_{\nu \in \bbZ} 
	|2^j\langle f,\widetilde{h}_{j,\nu}\rangle|^p\Big)^{1/p}\,=:\,A,
$$
and when $p=1$ then $f'$ is a finite Borel measure and $\|f'\|_{\cM}\lc A.$
 Consider the multiresolution analysis in $L_2(\SR)$ generated by the subspaces 
$$
V_N = {\overline{\text{span}}}\,\Big\{\cN_{N,\mu}:=\cN_2(2^{N+1}\cdot-\mu)~:~\mu \in \bbZ\Big\}\quad,\quad N=-1,0,1,2,...$$
That is, $V_N$ consists of continuous piecewise linear functions with nodes in $2^{-N-1}\SZ$.
Let $\cN^\ast(\cdot)$ be the (polygonal) function which generates the dual Riesz basis to $\{\cN_{N,\mu}~:~\mu \in \bbZ\}$; 
see e.g. \cite[\S3]{ChuiWang92}. 
Then, the operator
\[
h\in L_2\longmapsto P_N(h):=\sum_{\mu\in\SZ}2^N\hdot{h,\cN^{\ast}_{N,\mu}}\cN_{N,\mu}
\]
is the orthogonal projection onto $V_N$.
Let 
\[
	g_N(x):= \sum\limits_{\mu \in \bbZ} 2^N\langle f' , \cN_{N,\mu} \rangle\cN^{\ast}_{N,\mu}(x)\,.
\]
Using Lemma \ref{L_Nder},\footnote{Note that, in view of Proposition \ref{prop:test-funct},
one can give a meaning to the identity 
$\lan {f'}{\cN_{N,\mu}} \,=\,-2^{N+1}\,\lan{f}{\th_{N,\mu}}$ in Lemma \ref{L_Nder} also
for distributions $f\in\sB$.}  we have the uniform bound
\[
2^{-\frac Np}\Big(\sum\limits_{\mu \in \bbZ} |2^N\langle f', \cN_{N,\mu}\rangle |^p\Big)^\frac1p\\
	\lesssim 2^{N(1-\frac1p)}\Big(\sum\limits_{\mu \in \bbZ} |2^N\langle f, \widetilde{h}_{N,\mu}\rangle |^p\Big)^\frac1p\,\leq A<\infty.
\]
So, the exponential decay of $\cN^{\ast}(\cdot)$ guarantees that the series defining $g_N(x)$ converges, and moreover   
\begin{equation}
  \begin{split}
	\|g_N\|_p &\lesssim 2^{-N/p}\Big(\sum\limits_{\mu \in \bbZ} |2^N\langle f', \cN_{N,\mu}\rangle |^p\Big)^{1/p}\,\leq A.\\
  \end{split}
\end{equation}
Then, if $p>1$ there exists $g\in L_p$ which is the weak $^*$-limit of a subsequence of $\{g_N\}$.
Now, if $j,\nu$ are fixed, for all $N\geq j$ we have
\Be
\label{ggN}
\lan {g_N}{\psi_{j,\nu}}=\Hdot{f',\sum_{\mu\in\SZ}2^N\hdot{\psi_{j,\nu},\cN^{\ast}_{N,\mu}}\cN_{N,\mu}}=\lan {f'}{\psi_{j,\nu}},
\Ee
because $P_N(\psi_{j,\nu})=\psi_{j,\nu}$.
Thus, taking limits as $N\to\infty$ in \eqref{ggN} we obtain
\[
\hdot{g,\psi_{j,\,\nu}}=\hdot{f',\psi_{j,\,\nu}},\quad \mbox{for all $j,\nu$}.
\]
This implies that $g=f'\in L_p$ and $\|f'\|_p\lc A$.

When $p=1$ the weak* sequential-compactness argument  only provides that $\|g\|_\cM\lc A$ and then $\|f'\|_\cM\lc A$.
\qed 

\section{Embeddings into $B^{s,\dyad}_{p,\infty}$: The cases $s=1$ and $s=1/p-1$}\label{sec:embintoBdyad}

In this section we prove the sufficiency of the conditions for the  embeddings into $B^{1,\dyad}_{p,\infty}$ or $F^{1,\dyad}_{p,\infty} $ in Theorem \ref{emb_s=1} and the sufficiency for the conditions of embedding into $B^{1/p-1,\dyad}_{p,\infty} $ in Theorem \ref{emb_1/p-1}.


\begin{lemma}\label{lem:emb-min-p1}
Let $1/2\le p\le\infty$, $\frac 1p-1\le s\le 1$. Then
$ B^s_{p,\min\{p,1\}} \hookrightarrow B^{s,\dyad}_{p,\infty} $ is a continuous embedding.
\end{lemma}

\begin{proof} Using the notation in the proof of part (ii) of  Proposition \ref{prop1}
we can write
\[
\|f\|_{B^{s,{\rm dyad}}_{p,\infty}}=\sup_{j\geq-1} 2^{js}\,B_j
\] 
where $B_j$ is defined  as in \eqref{Bj} with $\delta=0$. 
Letting $u=\min\{p,1\}$, we obtain for each $j$, arguing as in relation \eqref{f41}
\Beas
2^{js}\,B_j & \lesssim & 
\Big(\sum\limits_{\ell\in\SZ}\Big[a(\ell,\tfrac1p)\,2^{js}\,\big\|\fM_{j+\ell}(f_{j+\ell})\big\|_p\Big]^u\Big)^{1/u}\\
& \lesssim & 
\Big(\sup_{\ell\in\SZ}\, a(\ell,\tfrac1p)2^{-\ell s}\Big) \,\|f\|_{B^s_{p,u}},
\Eeas
where  $a(\ell, 1/p)=2^{(1/p-1)\ell}$ for $\ell>0$ and $a(\ell,1/p)=2^\ell$ for $\ell<0$. Since
$\sup_{\ell\in\SZ}\, a(\ell,1/p)2^{-s\ell} < \infty
$
whenever $1/p-1\leq s \leq 1$,  we obtain the desired inequality 
$    \|f\|_{B^{s,\mathrm{dyad}}_{p,\infty}} \lesssim \|f\|_{B^s_{p,\min\{p,1\}}} $.
\end{proof}

\begin{proposition} \label{prop:embintoF1dyad}
Let $1/2<p<\infty$. Then
$F^1_{p,2} \hookrightarrow F^{1,\dyad}_{p,\infty}$ is a continuous embedding.
\end{proposition}
\begin{proof}
Let $f\in F^1_{p,2}$. we must show that
\[
\|f\|_{F^{1,{\rm dyad}}_{p,\infty}}=\big\|\sup_{j\geq-1} 2^j\sum_{\nu\in\SZ}|2^j\langle f,h_{j,\nu}\rangle|
\,\bbone_{I_{j,\nu}}\big\|_p \lesssim \|f\|_{F^1_{p,2}}.
\]
With the notation from Section \ref{sect:MF_FS}, we write $f=\sum_{k\geq0}L_kf_k$, and for each $j\geq-1$ we 
split $f=\Pi_jf +\Pi_j^\perp f$
where
\[
\Pi_jf:=\sum_{k=0}^jL_kf_k,\mand \Pi_j^\perp f :=\sum_{k>j}L_kf_k.
\]
If $j=-1$, we understand that $\Pi_{-1}f=0$ and $\Pi^\perp_{-1}f=f$.
We shall first bound
\[
A_0:=\big\|\sup_{j\geq0} 2^j\sum_{\nu\in\SZ}|2^j\langle \Pi_jf,h_{j,\nu}\rangle|
\,\bbone_{I_{j,\nu}}\big\|_p.
\]
The same argument in the proof of \cite[Lemma 3.3]{GaSeUl21_2} gives
\[
|2^j\langle \Pi_jf,h_{j,\nu}\rangle|\lesssim\fM^{**}_{j,A}\big(2^{-j}(\Pi_jf)'\big)(x), \quad x\in I_{j,\nu}.
\]
Thus, 
\[
2^j\sum_{\nu\in\SZ}|2^j\langle \Pi_jf,h_{j,\nu}\rangle|\,\bbone_{I_{j,\nu}}(x)\lesssim \fM^{**}_{j,A}\big((\Pi_jf)'\big)(x), \quad x\in\SR,
\]
and taking a supremum  over all $j\geq0$, and then $L_p$-norms, we obtain
\Be
A_0\,\lesssim\, 
\big\|\sup_{j\geq0} \fM^{**}_{j,A}\big((\Pi_jf)'\big)\big\|_p\,
\lesssim \big\|\sup_{j\geq0} |(\Pi_jf)'|\big\|_p\,,
\label{f00}
\Ee
 using \eqref{Peetremax} in the last step with $A>1/p$.
Now, the maximal function characterization of the $h^p=F^0_{p,2}$ norms yields 
\Be
    \big\|\sup_{j \geq 0} |(\Pi_jf)'|\big\|_p = \big\|\sup_{j\geq0} |\Pi_j(f')|\big\|_p \lesssim \|f'\|_{F^0_{p,2}} 
		\,\lesssim\,\|f\|_{F^1_{p,2}}.
    \label{perp0}
\Ee
To estimate the remaining part involving $\Pi_j^\perp f$, we may 
quote 
the standard proof in Proposition \ref{prop1}\footnote{That is, the part of the proof of
 Proposition \ref{prop1} involving the indices $\ell=k-j\geq0$.} above, which gives
\Be\label{perp1}
A_1:=\big\|\sup\limits_{j\geq -1} 2^j\sum_{\nu\in\SZ}|2^j\langle \Pi_j^\perp f,h_{j,\nu}\rangle|\,\bbone_{I_{j,\nu}}(x)\big\|_p 
\lesssim \|f\|_{F^1_{p,\infty}},
\Ee
provided that $s>\max\{1/p,1/q\}-1$, with $s=1$ and $q=\infty$. So, we obtain
\[
A_1\,\lesssim\, \|f\|_{F^1_{p,\infty}}\lesssim  \|f\|_{F^1_{p,2}},
\]
under the assumption that $p>1/2$. Finally, \eqref{f00}, \eqref{perp0}, \eqref{perp1} yield the desired estimate 
        $\|f\|_{F^{1,\mathrm{dyad}}_{p,\infty}} \lesssim \|f\|_{F^1_{p,2}}
$.
\end{proof}

\begin{corollary} \label{cor:B1-qle2emb}
Let $1/2\le p<\infty$  and $q\le \min\{p,2\}$. 
Then $B^1_{p,q}\hookrightarrow B^{1,\dyad}_{p,\infty}$ is a continuous embedding.
\end{corollary}
\begin{proof}
For $q\le p\le 1$ this follows from Lemma \ref{lem:emb-min-p1}.
For $1/2<p<\infty$  it  follows from Proposition \ref{prop:embintoF1dyad} together with  the inequality
\Be\label{FdyadvesrsusBdyad} 
\|f\|_{B^{1,\dyad}_{p,\infty}} \lc \|f\|_{F^{1,\dyad}_{p,\infty} }
\Ee 
and the inequality 
\[\|f\|_{F^1_{p,2}} \lc \|f\|_{B^1_{p,\min \{p,2\}}};\]
the latter being a consequence of Minkowski's inequalities.

Inequality \eqref{FdyadvesrsusBdyad} in turn follows by definition from the  sequence space inequality 
$\|\beta\|_{b^s_{p,\infty} }\le \|\beta\|_{f^s_{p,\infty} }$
(in the case $s=1$), i.e. from the elementary inequality
\begin{multline} \label{eq:fbersusbsequenceemb}
\sup_j   \Big(\sum_{\mu\in \bbZ  } 
\int_{I_{j,\mu}} | 2^{js} \beta_{j,\mu} \bbone_{I_{j,\mu}}(x)  |^p dx\Big)^{1/p} 
\\
\le \Big(\int \Big[ \sup_j 2^{js} \Big|\sum_{\mu\in \bbZ} \beta_{j,\mu} \bbone_{I_{j,\mu}}(x) \Big|\Big]^p dx\Big)^{1/p} .\qedhere\end{multline} 
\end{proof}

We now consider the other limiting case, where $s=1/p-1$.
\begin{proposition} \label{prop:F1p-1dyadtoB}
Let $1/2<p\le 1$ and $0<q\le\infty$. Then
$F^{\frac 1p-1}_{p,q} \hookrightarrow B^{\frac 1p-1,\dyad}_{p,\infty}$ is a continuous embedding. 
\end{proposition}


\begin{proof} Since $F^{\frac 1p-1}_{p,q} \hookrightarrow F^{\frac 1p-1}_{p,\infty} $ it suffices to prove this for $q=\infty$. 

Let $f\in F^{1/p-1}_{p,\infty}$, which as before we shall split as 
\[
f=\Pi_jf+ \Pi_j^\perp f.
\]
This time, the standard proof in Proposition \ref{prop1} (that is,  the part of the proof 
 involving the indices $\ell=k-j\leq 0$)  gives
\[
\Big\|\Big(\sum\limits_{j\geq0} 2^{jsq}\sum_{\nu\in\SZ}|2^j\langle \Pi_j f,h_{j,\nu}\rangle|^q\,\bbone_{I_{j,\nu}}\Big)^{1/q}\Big\|_{L_p(\bbR)} \lesssim \|f\|_{F^s_{p,q}}
\]
provided that $s<1$. So in particular,  letting $s=1/p-1$ and $q=\infty$ we obtain (after trivial embeddings)
\Be\label{f12_b}
\sup_{j\geq0} 2^{j(1/p-1)} \Big\|\sum_{\nu\in\SZ}2^j\langle \Pi_j f,h_{j,\nu}\rangle\,\bbone_{I_{j,\nu}}\Big\|_{L_p(\bbR)}
\lesssim  \|f\|_{F^{\frac1p-1}_{p,\infty}},
\Ee
whenever $p>1/2$. 

So, it remains to establish a similar estimate with $\Pi_j^{\perp}f$ instead of $\Pi_j f$.
We borrow some notation from \cite{GaSeUl21_2}. Let $\cD_j$ denote the dyadic intervals of length $2^{-j}$.
If $I\in\cD_j$ is fixed and $k\geq j$ we let
\Benu
\item[(a)] $\cD_k(\partial I)=\big\{J\in\cD_k\mid \bar{J}\cap \partial I\not=\emptyset\big\}$
\item[(b)] $\om(J)=\{x\in J\mid \dist(x,\partial I)\geq 2^{-k-1}\}$, when $J\in\cD_k(\partial I)$.
\Eenu
We will use  the maximal function $\fM^*_kg$  
(\cf. \eqref{eq:fMk*})
and note that, as in \cite[(42)]{GaSeUl21_2}, when $J\in \cD_{k}(\partial I)$, $k>j$, it holds
\Be
\sup_{x\in J}|g(x)|\leq \Big[
\mint_{\om(J)} |\fM^*_kg|^p\,\Big]^{\frac1p}, \quad 0<p<\infty.
\label{gLp}
\Ee
Now, let $I=I_{j,\nu}\in\cD_j$ be fixed, and let $I^\pm$ be its dyadic sons. For each $k>j$,
the function $L_k(h_{j,\nu})$ has support contained in the union of the intervals $J$ belonging to $\cD_k(\partial I^\pm)$,
and $|L_k(h_{j,\nu})|\lesssim 1$. Thus, since we are assuming $p\leq1$, we have
\Beas
\big|\langle 2^j h_{j,\nu},\Pi^\perp_jf\rangle \big|^p & \leq & \sum_{k>j}\big|\langle L_k(2^jh_{j,\nu}),f_k \rangle\big|^p\\
& \lesssim & \sum_{k>j}2^{-(k-j)p}\sum_{J\in\cD_k(\partial I^\pm)}\|f_k\|_{L_\infty(J)}^p
\Eeas
which, by \eqref{gLp}, is bounded by
\[\sum_{k>j}2^{-(k-j)p}\sum_{J\in\cD_k(\partial I^\pm)}
\mint_{\om(J)} |\fM^*_kf_k(x)|^p \lesssim  2^{jp}\,\int_{I^*}\sup_{k>j}\big|2^{(\frac1p-1)k}\,\fM^*_kf_k\big|^p\,dx\,,
\]
where $I^*$ is the 2-fold dilation of the interval $I$.
Summing up in all intervals $I=I_{j,\nu}\in\cD_j$ we obtain
\[
\sum_{\nu\in\SZ}\big|\langle 2^j h_{j,\nu},\Pi^\perp_jf\rangle \big|^p\lesssim 2^{jp}\,
\int_\SR\sup_{k>j}\big|2^{(\frac1p-1)k}\,\fM^*_kf_k\big|^p.
\]
This implies
\Be\label{f12_c}
\sup_j 2^{-j}\Big(\sum_{\nu\in\SZ}\big|\langle 2^j h_{j,\nu},\Pi^\perp_jf\rangle \big|^p\Big)^{1/p}\lesssim 
\Big(\int_\SR\sup_{k>j}\big|2^{(\frac1p-1)k}\,\fM^*_kf_k\big|^p\Big)^{1/p}\,.
\Ee
Now, $\fM^*_kg\lesssim \fM_{k,A}^{**}g$ for any $A>0$, so choosing $A>1/p$ and using \eqref{Peetremax}
we have
\Be\label{f12_d}
\big\|\sup_{k\geq0}2^{(\frac1p-1)k}\,\fM^*_kf_k\big\|_p\lesssim \big\|\sup_{k\geq0}2^{(\frac1p-1)k}\,|f_k|\big\|_p
\lesssim \|f\|_{F^{\frac1p-1}_{p,\infty}}.
\Ee
Finally, the  inequality 
\[\|f\|_{B^{ \frac 1p-1,\dyad}_{p,\infty} }\lc\|f\|_{F^{\frac 1p-1}_{p,\infty}} \]
follows by combining \eqref{f12_b}, \eqref{f12_c} and \eqref{f12_d}.
\end{proof}

We now consider the remaining endpoint case  where the two limiting cases coincide, that is,   we have both $s=1$  and $s=1/p-1$, and thus  $p=1/2$ (the corresponding Besov embedding is already covered in 
Lemma \ref{lem:emb-min-p1}).

\begin{proposition} \label{s=1=1/p-1}
For $p=1/2$ we have the continuous embedding \[F^1_{  1/2,2} \hookrightarrow B^{1,\dyad}_{ 1/2,\infty}
.\]
\end{proposition}
\begin{proof} 
We examine the proof of Propositions \ref{prop:F1p-1dyadtoB} and \ref{prop:embintoF1dyad} and note 
that \eqref{f12_c} and \eqref{f12_d} remain valid for $p=1/2$,
that is
\[
   \sup_{j \geq -1} 2^{-j}\Big(\sum_{\nu\in\SZ}\big|\langle 2^j h_{j,\nu},\Pi^\perp_jf\rangle \big|^\frac 12 \Big)^{2}\lesssim 
\|f\|_{F^{1}_{ 1/2,\infty}}.
\]
Similarly, the arguments in \eqref{f00} and \eqref{perp0} do not require a restriction on $p$, so we also have 
\[
\big\|\sup_{j\geq 0}2^j\sum\limits_{\nu\in \bbZ}|2^j\langle \Pi_j f,h_{j,\nu}\rangle| h_{j,\nu}\big\|_{1/2} \lesssim \|f\|_{F^1_{ 1/2,2}} .
\]
Now the proposition follows from 
the trivial embeddings $F^1_{1/2,2} \hookrightarrow F^1_{1/2,\infty} $ and $f^1_{1/2,\infty}\hookrightarrow b^1_{1/2,\infty}$
(\cf. \eqref{eq:fbersusbsequenceemb}).
 \end{proof}

\section{ Norm equivalences  on suitable subspaces:\\ The proofs of Theorems \ref{thm:Besov-equiv} and \ref{W1-equiv-intro}}
\label{sec:normequivalences}

\subsection{\it A bootstrapping lemma}

Consider a sequence $\fa=\{a_{j,\mu}\}$ indexed by $j\in \bbN\cup\{0\}$ and $\mu\in \bbZ$. 
As in \eqref{b_sequence} let $b^s _{p,q}$ be the set of all $\fa$ for which 
\Be 
\|\fa\|_{b^s _{p,q} }= \Big(\sum_{j=0}^\infty2^{j(s -\frac1p) q}\Big[\sum_{\mu\in\bbZ}|a_{j,\mu}|^p\Big]^{\frac qp}\Big)^{1/q}<\infty.
\Ee 
We split each sequence as $\fa=\faeven+\faodd$, where
\[
\fa^{\mathrm{even}}_{j,\mu} = \begin{cases} a_{j,\mu}
& \text{ if $\mu$ is even } \\0&\text{ if $\mu$ is odd,} 
\end{cases}  \mand
\fa^{\mathrm{odd}}_{j,\mu}= 
\begin{cases}0 
& \text{ if $\mu$ is even } \\a_{j,\mu}&\text{ if $\mu$ is odd.} 
\end{cases}  \]
Then  
\begin{align*}\big\|\faeven\big\|_{b^s _{p,q} }&= 
\Big(\sum_{j=0}^\infty2^{j(s -\frac1p) q}\Big[\sum_{\nu\in\bbZ}|a_{j,2\nu}|^p\Big]^{\frac qp}\Big)^{1/q}
\\
\big\|\faodd\big\|_{b^s _{p,q} }&= 
\Big(\sum_{j=0}^\infty2^{j(s -\frac1p) q}\Big[\sum_{\nu\in\bbZ}|a_{j,2\nu+1}|^p\Big]^{\frac qp}\Big)^{1/q}
\end{align*}
The key result is the following  lemma,
 which, under suitable conditions, allows us to control $\|\fa\|_{b^s _{p,q}}$ in terms of $\|\faeven\|_{b^{s }_{p,q}}$. 
The general hypothesis in \eqref{la} will be linked later to a refinement condition 
which will appear in \eqref{refine1}.

\begin{lemma}\label{L1}
Let $0<p, q\leq\infty$, $s \in\bbR$ and $\vec \la=(\la_0,\la_1,\la_2)\in\bbC^3$ such that
\Be
\label{la1sig}
|\la_1|<2^{s -\frac1p}.
\Ee
Then, there exists $C=C(p,q,s ,\vec\la)>0$ such that for every sequence $\fa\in b^s _{p,q}$ satisfying the condition 
\Be
\label{la}
|a_{j,2\nu+1}|\leq \sum_{\ell=0}^2 |\la_\ell| \,|a_{j+1, 4\nu+2+\ell}|, 
\text{ for all } j\ge 0,\, \nu \in \bbZ.
\Ee
we have 
\Be
\label{aeven}
\big\|\fa\big\|_{b^s _{p,q}}\leq \, C\, \big\|\faeven\big\|_{b^s _{p,q}}.
\Ee
\end{lemma}
\begin{proof} 
Let $\rho=\min\{1,p,q\}$, and for simplicity write $\sigma=s -1/p$. Condition \eqref{la} together with the $\rho$-triangle inequality gives 
\begin{align*}\big\|\faodd\big\|_{b^s _{p,q}}^\rho&\le \sum_{\ell=0,1,2}|\la_\ell|^\rho
\Big(\sum_{j=0}^\infty 2^{j\sigma q } \Big[\sum_{\nu\in \bbZ}  
|a_{j+1, 4\nu+2+\ell} |^p\Big]^{q/p} \Big)^{\rho/q} 
\\
&= (|\la_1| 2^{-\sigma})^\rho \Big(\sum_{j=0}^\infty 2^{ (j+1)\sigma q } \Big[\sum_{\nu\in \bbZ}  
|a_{j+1, 4\nu+3} |^p\Big]^{q/p} \Big)^{\rho/q} 
\\&\quad +
\sum_{\ell=0,2} (|\la_\ell| 2^{-\sigma})^{\rho} 
\Big(\sum_{j=0}^\infty 2^{ (j+1) \sigma q } \Big[\sum_{\nu\in \bbZ}  
|a_{j+1, 4\nu+2+\ell} |^p\Big]^{q/p} \Big)^{\rho/q}.
\end{align*}
Using the assumptions $|\la_1|< 2^\sigma$ and $\fa\in b^s _{p,q}$, the previous display implies
\[ \Big(1- \frac{|\la_1|^\rho }{2^{\sigma\rho} }\Big)\,
 \big\|\faodd\big\|_{b^s _{p,q}}^\rho  \le  (|\la_0|^\rho+|\la_2|^\rho) 2^{-\sigma\rho} \|\faeven \|_{b^s _{p,q}}^\rho.
\]
This gives 
\begin{align*} \|\fa\|_{b^s _{p,q} } &\le \Big( \|\faodd \|_{b^s _{p,q}} ^\rho+\|\faeven \|_{b^s _{p,q}}^\rho \Big)^{1/\rho} 
\\
&\le \Big(\frac{|\la_0|^\rho+|\la_2|^\rho} 
{2^{\sigma\rho}- |\la_1|^\rho}+1\Big)^{1/\rho }  \|\faeven\|_{b^s _{p,q}},
\end{align*}
which finishes the proof. 
\end{proof} 

\subsection{\it Proof of  Theorem \ref{thm:Besov-equiv}}
We must show \eqref{BBd}, that is
\Be
\label{direct}
\|f\|_{B^s_{p,q}} \lesssim 
\|f\|_{B^{s,{\rm dyad}}_{p,q}} ,
\quad \mbox{provided $f\in B^s_{p,q}$},\Ee
which, as we shall see, holds  actually in the larger range 
\[
\frac1p<s<1+\frac1p.\]
By part (ii) of Proposition \ref{main1}, we know that
\[
\|f\|_{B^s_{p,q}} \lesssim \big\|\{\fc_{j,\mu}(f)\}_{ {j\geq0, \mu\in\SZ}}\big\|_{b^s_{p,q}} 
\,+\,\big\|\big \{\lan f{h_{-1,\mu}}\big\}_{\mu\in\SZ}\big\|_{\ell_p}=A+B.
\]
Clearly, $B$ is bounded by the right hand side of \eqref{direct}, so we focus on $A$.
Define a sequence $\fa=\fa(f)=\{a_{j,\nu}\}$ by
\[
a_{j,\nu}=2^j|\lan{f}{\th_{j,\nu}}|, \quad j\in \SN_0, \;\nu\in\SZ.
\]
Observe from the definition of the coefficients $\fc_{j,\mu}(f)$ in \eqref{ext-Haar-coeff} that
\[
A\,\lesssim\, \|\fa\|_{b^s_{p,q}}.
\]
Note also from \eqref{thjnu} that $a_{j, 2\mu}=2^j|\lan f{h_{j,\mu}}|$, so in particular
\[
\|\faeven\|_{b^s _{p,q}}=\big\|\big\{ 2^j\lan f{h_{j,\mu}}\} _{j\geq0,\mu\in\SZ}\big\|_{b^s_{p,q}}\leq \|f\|_{B^{s,{\rm dyad}}_{p,q}}.
\]
Therefore, we have reduced matters to prove that
\[
\big\|\fa\big\|_{b^s _{p,q}}\lesssim\, \big\|\faeven\big\|_{b^s _{p,q}}.
\]
We shall do so using Lemma \ref{L1}, so we need to verify the hypothesis \eqref{la}, for suitable 
scalars $(\la_0,\la_1,\la_2)$. This will follow from an elementary property of the spline functions
\[
\cN_{j,\mu}(x):=\cN_{2;j,\mu}(x)=\cN_2(2^jx-\mu),
\]
defined in \S\ref{S_chui}. Recall that these are piecewise linear functions supported in the intervals
$[2^{-j}\mu, 2^{-j}(\mu+2)]$. It is then straightforward to verify that 
\Be
\cN_{j,\mu }(x)=\tfrac12\,\cN_{j+1, 2\mu }(x)+\,\cN_{j+1, 2\mu +1}(x)+\tfrac12\,\cN_{j+1, 2\mu +2}(x);
\label{refine1}
\Ee
see Figure \ref{fig_refine}. We refer to \eqref{refine1} as the \emph{refinement identity}.

\begin{figure}[h]
 \centering
{
\begin{tikzpicture}[scale=4]

\draw[->] (-0.7,0.0) -- (1.5,0.0) node[below] {$x$};
\draw[->] (-0.5,-0.2) -- (-0.5,1.1);


\draw[dotted] (0.5,1) --(-0.5,1) node[left] {$1$};
\draw[dotted] (0.25,0.5) --(-0.5,0.5) node[left] {$1/2$};

\draw[red] (0,0)--(.5,1)--(1,0);
\draw[dashed, very thick, blue] (0,0)--(.25,0.5)--(0.5,0);
\draw[dashed, very thick, blue] (.25,0)--(0.5,1)--(0.75,0);
\draw[dashed, very thick, blue] (0.5,0)--(.75,0.5)--(1,0);


\fill[black] (0,0) circle (0.015cm)  node [below] {$\tfrac\nu{2^{j}}$};
\fill[black] (0.5,0) circle (0.015cm)  node [below] {$\tfrac{\nu+1}{2^{j}}$};
\fill[black] (1,0) circle (0.015cm)  node [below] {$\tfrac{\nu+2}{2^{j}}$};

\fill[blue] (0.25,0) circle (0.015cm);
\fill[blue] (0.75,0) circle (0.015cm);


\node [right] at (0.65, 0.8 ) {{\color{red} $\cN_{j,\nu}$}};

\node [left] at (0.15, 0.35 ) {{\color{blue} $\tfrac12\cN_{j+1,2\nu}$}};
\node [right] at (0.85, 0.35 ) {{\color{blue} $\tfrac12\cN_{j+1,2\nu+2}$}};

\node [left] at (0.25, 0.8 ) {{\color{blue} $\cN_{j+1,2\nu+1}$}};
\draw[->] (0.1, 0.75 ) -- (0.4,0.65);

\end{tikzpicture}
}

\caption{Refinement equation for $\cN_{j,\nu}(x)$; see \eqref{refine1}.
}\label{fig_refine}
\end{figure}
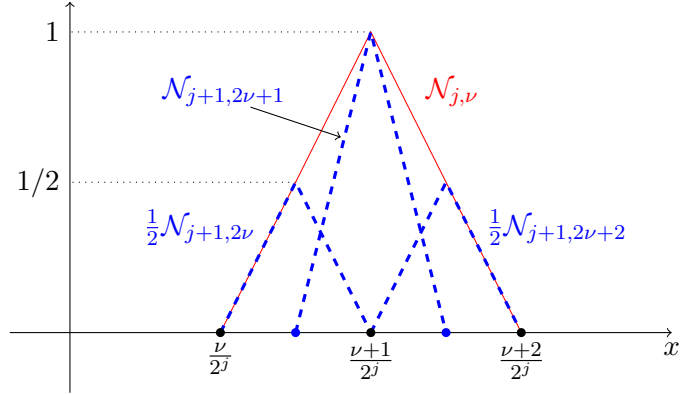

Now, if $\mu=2\nu+1$ is an odd integer, then the integration by parts formula and the refinement identity give
\Beas
a_{j,\mu} & = & 2^j\,|\hdot{f, \th_{j,\mu}}|=\frac12\,|\hdot{ f',\cN_{j+1,\mu}}|\\
\mbox{{\tiny by \eqref{refine1}}} & \leq & \frac14\,|\hdot{  f',\cN_{j+2,2\mu}}|+
\frac12\,|\hdot{ f',\cN_{j+2,2\mu+1}}|+\frac14\,|\hdot{  f',\cN_{j+2,2\mu+2}}|\\
\mbox{{\tiny by \eqref{Nder}}} & = & \frac12\,a_{j+1,2\mu}+
\,a_{j+1,2\mu+1}+\frac12\,a_{j+1,2\mu+2},
\Eeas
which coincides with \eqref{la} with $(\la_0,\la_1,\la_2)=(1/2, 1, 1/2)$. 
So we can apply Lemma \ref{L1}, under the assumption
\[
|\la_1|=1<2^{s-\frac1p},
\]
which holds precisely when $s>\frac1p$. This completes the proof.
\qed

\subsection{\it Proof of Theorem \ref{W1-equiv-intro}}
We first show that, if $1<p\leq \infty$ and  $f\in W^{1}_p(\bbR)$, then
$$
	\|f\|_{W^{1}_p} \approx \sup\limits_{j \geq -1} 2^{j(1-1/p)}\Big(\sum\limits_{\mu \in \bbZ}|2^j\langle f,h_{j,\mu}\rangle|^p\Big)^{1/p}\,.
$$
In view of Theorem \ref{W1-char-intro}, this reduces to prove
\begin{equation}\label{bootstrap}
  \begin{split}
	A &:= \sup\limits_{j\geq -1} 2^{j(1-1/p)}\Big(\sum\limits_{\mu \in \bbZ} 
	|\fc_{j,\mu}(f)|^p\Big)^{1/p}
	\\
	&\lesssim \,\sup\limits_{j\geq -1} 2^{j(1-1/p)}\Big(\sum\limits_{\mu \in \bbZ} |2^j\langle f,h_{j,\mu}\rangle|^p\Big)^{1/p}\,,
  \end{split}
\end{equation}
whenever $A$ is finite. The argument is completely analogous to that for proving Theorem \ref{thm:Besov-equiv}\,,
this time using the spaces $b^1_{p,\infty}$. 
Note, that this argument only needs that $1=s>1/p$. 

Finally, the assertion in \eqref{W1BF}, for $f\in W^1_p$, follows now easily from 
\[
\|f\|_{W^1_p}\lesssim \|f\|_{B^{1,{\rm dyad}}_{p,\infty}}\lesssim \|f\|_{F^{1,{\rm dyad}}_{p,\infty}}\lesssim 
\|f\|_{F^1_{p,2}}\equiv \|f\|_{W^1_p},
\]
where the third inequality was shown in part (ii) of Theorem \ref{emb_s=1} and for the last three steps we assume $p<\infty$.
\qed

\section{Necessary condition for the embeddings for $s=1$}\label{sec-s=1-example}

We prove the necessary conditions in Theorem \ref{emb_s=1} for various embeddings into $B^{1,\dyad}_{p,\infty}$.

\begin{lemma} \label{lem:ex1_B}
 Suppose $1/2\le p<\infty$.
Then
\Be  B^1_{p,q} \hookrightarrow B_{p,\infty}^{1,\dyad} \implies  q\le  p\label{eq:ex1_B}
\Ee
\end{lemma}
\begin{proof}We shall work with an example that has been used in \cite[\S6.2]{GaSeUl21} to prove lower bounds for the norms of $\SE_N$ on $B^1_{p,q}$. Let $u\in C^\infty_c$ be supported in $(1/8,7/8)$ so that $u(x)=1$  on $[1/4,3/4]$. For $N\gg 1$ and $N/4\le j\le N/2$ define 
\[g_{N,j}(x)= u\big(N(x-\tfrac{2j}{N})\big) e^{2\pi i 2^j x} \] and let 
$f_N(x) =\sum_{N/4\le j\le N/2} 2^{-j} g_{N,j}(x)$. Then by \cite[Lemma 29]{GaSeUl21} (Lemma 6.3 in arxiv:1901:09117) we have $\|f_N\|_{B^1_{p,q} }\lc N^{-(1/p-1/q)}$ for $p\leq q$. 
We show that $\|f_N\|_{B^{1,\dyad}_{p,\infty}} \gc 1$ for large $N$, which will imply that $B^1_{p,q}$ is not continuously embedded into $B^{1,\dyad}_{p,\infty}$ when $q>p$. 

To see this we prove lower bounds for many of the   Haar coefficients of $f_N$ at Haar frequency $2^N$.
 Let $J^{N,j}=(\frac{2j}{N}+\frac 1{4N}, \frac{2j}{N}+\frac{3}{4N})$; we observe that for fixed $N$ the  intervals $J^{N,j}$ are disjoint and that $f_N(x)=e^{ 2^j 2\pi i x} $ for $x\in J^{N,j}$.  
We get by a Taylor expansion
\begin{align}\inn{f_N}{h_{N,\mu}} &=2^{-2N-2} f_N'\big(2^{-N}(\mu+\tfrac 12)\big)  + R_{N,\mu}
\label{taylor}
\end{align} with $|R_{N,\mu}|\le 2^{-3N} \sup_{I_{j,\mu}} |f_N''|$.
 Let $\cZ^{N,j}$ be the set of all integers $\mu$ such that
$2^{-N}\mu$ and $2^{-N}(\mu+1)$ belong to $J^{N,j}$; then 
for $\mu\in \cZ_{N,j}$\[|\inn {f_N}{h_{j,\mu}}|=|\inn { 2^{-j} g_{N,j}}{h_{j,\mu}}|= 2\pi\,2^{-2N-2}+ O(2^{j-3N})\] and hence
\[\|f_N\|_{B^{1,\dyad}_{p,\infty}} \ge 2^{N(1-1/p)} \Big(
\sum_{\frac N4<j<\frac N2} \sum_{\mu\in \cZ_{N,j}} 
\big|2^N\inn {f_N}{h_{j,\mu}}\big|^p\Big)^{1/p}.
\]
Since 
$\#(\cZ^{N,j}) \approx 2^N N^{-1}$
for large $N$ and $N/4<j<N/2$ we obtain $\|f_N\|_{B^{1,\dyad}_{p,\infty}} \gc 1$.
\end{proof} 

\begin{lemma} \label{lem:ex1_F}
 Suppose $1/2\le p<\infty$.
Then
\begin{align}
F^1_{p,q} \hookrightarrow B_{p,\infty}^{1,\dyad} \implies  q\le 2, \label{eq:ex1_F}
\\
B^1_{p,q} \hookrightarrow B_{p,\infty}^{1,\dyad} \implies  q\le 2.\label{eq:ex1_B2}
\end{align} 
\end{lemma}
\begin{proof} We consider the same example that was used in \cite[\S7.2.1]{GaSeUl21_2}. Namely, let $\psi\in C^\infty_c(0,1)$ with $\psi\equiv1$ in $[1/4,3/4]$, and 
for each large $N\gg1$, let \Be\label{eq:cZN} \fZ_N=\{j\in\SN\mid N/4< j<N/2\}
\Ee
and
\Be\label{ft}
f_t(x):=\sum_{j\in\fZ_N}\frac{r_j(t)}{2^j}e^{2\pi i 2^jx}\,\psi(x),\quad t\in[0,1],
\Ee
where $r_j(t)$, $t\in[0,1]$, are the usual Rademacher functions.
Using Lemma 7.3 from \cite{GaSeUl21_2} one can verify that
\Be\label{ftF}
\sup_{t\in [0,1]} \|f_t\|_{F^1_{p,q}} 
\lesssim N^{1/q};
\Ee
a similar argument also gives
\Be\label{ftB}
\sup_{t\in [0,1]} \|f_t\|_{B^1_{p,q}}
\lesssim N^{1/q}.
\Ee
Let $\psi_j(x)=2^{-j}e^{2\pi i 2^jx}\psi(x)$, for $j\in\fZ_N$, and let $\cZ_N$ be the set of all $\mu\in\SZ$ such that  $I_{N,\mu}\subset(1/4,3/4)$. Using the Taylor expansion as in \eqref{taylor} one sees that
\Be
|\hdot{\psi_j, 2^Nh_{N,\mu}}|=2\pi 2^{-N-2}+O(2^{j-2N}),\quad \mu\in\cZ_N.
\label{psihj}
\Ee
Now observe that
\Be\label{ftdyad}
\|f_t\|_{B^{1,{\rm dyad}}_{p,\infty}}\geq 2^{N(1-1/p)}\,\Big(\sum_{\mu\in\cZ_N}|\lan{f_t}{2^Nh_{N,\mu}}|^p\Big)^\frac1p.
\Ee
So, raising to the $p$-th power and taking the expectation 
in the $t$ variable, we obtain from Khintchine's inequality
\Beas
\Big(\int_0^1\|f_t\|_{B^{1,{\rm dyad}}_{p,\infty}}^pdt\Big)^\frac1p &\geq  & 2^{N(1-\frac1p)}\,\Big(\sum_{\mu\in\cZ_N}\int_0^1 \big|\sum_{j\in\fZ_N}r_j(t)\lan{\psi_j}{2^Nh_{N,\mu}}\big|^pdt\Big)^\frac1p
\\
&\gtrsim  & 2^{N(1-\frac1p)}\,\Big(\sum_{\mu\in\cZ_N}
\Big[
 \sum_{j\in\fZ_N}\big|\lan{\psi_j}{2^Nh_{N,\mu}}\big|^2\Big]^\frac p2\Big)^\frac1p.
\Eeas
An application of \eqref{psihj}, together with the cardinalities of $\fZ_N$ and $\cZ_N$, then gives
\[
\Big(\int_0^1 \|f_t\|_{B^{1,{\rm dyad}}_{p,\infty}}^p dt \Big)^\frac1p\,\gtrsim\, \sqrt{N}.
\]
This together with \eqref{ftF}, \eqref{ftB}  implies that the inclusions $B^1_{p,q}\hookrightarrow B^{1,{\rm dyad}}_{p,\infty}$
and 
$F^1_{p,q}\hookrightarrow B^{1,{\rm dyad}}_{p,\infty}$  can only hold if $q\leq 2$.
\end{proof}

\begin{lemma}\label{lem:q=infty}
For $p=\infty$ we have 
\Be\label{q=infty} B^1_{\infty,q} \hookrightarrow B_{\infty,\infty}^{1,\dyad} \implies q\leq1.
\Ee
\end{lemma}

\begin{proof}
We assume that $B^1_{\infty,q} \hookrightarrow B_{\infty,\infty}^{1,\dyad}$, and we shall prove that necessarily $q\leq1$. Let $\cZ_N$ be as in \eqref{eq:cZN} and consider the function
\[
f(x):=\sum_{j\in\fZ_N}2^{-j}e^{2\pi i 2^jx}\,\psi(x),
\]
which is defined as in \eqref{ft}, but with all the  $r_j(t)$ set equal to 1. This time we shall assume that $\psi\in C^\infty_c(-1/2,1/2)$ with $\psi=1$ in $(-1/4,1/4)$. As in \eqref{ftB} we have
\Be\label{eq:fNB1pqupper}
\|f\|_{B^1_{p,q}}
\lesssim N^{1/q}.
\Ee
On the other hand,  note that
\Be\label{fhN0}
\|f\|_{B^{1,{\rm dyad}}_{\infty,\infty}}\geq 2^{N}\,|\lan{f}{2^Nh_{N,0}}|.
\Ee
Arguing as in \eqref{taylor} we see that
\Beas
\lan{f}{2^Nh_{N,0}} & = & \sum_{j\in\fZ_N}\lan{\psi_j}{2^Nh_{N,0}}=\sum_{j\in\fZ_N}\big[2^{-j}2^{-N-2}\psi_j'(0) +O(2^{j-2N})\big]\\
& = & 2\pi i 2^{-N-2}\,{\rm Card\,}(\fZ_N)\, +O(2^{-3N/2}),
\Eeas
which inserted into \eqref{fhN0} gives
\Be\label{eq:fNBduadlower}
\|f\|_{B^{1,{\rm dyad}}_{\infty,\infty}}\gtrsim \,N.
\Ee
The lemma is proved after combining \eqref{eq:fNB1pqupper}, \eqref{eq:fNBduadlower} and letting $N\to \infty$.
\end{proof}

\section{ Necessary condition for embedding into $B^{1/p-1,\dyad}_{p,\infty} $ }
\label{sec:s=1/p-1-ex}

\begin{proposition} \label{prop:neq-1/p'} Let $0<p,q\leq\infty$.
Then
\[
B^{1/p-1}_{p,q}\hookrightarrow B^{1/p-1,\mathrm{dyad}}_{p,\infty} \,\implies \, q\le \min \{1,p\}.
\]
\end{proposition}
\begin{proof}
We first assume $1<p\leq\infty$. Suppose the embedding  $B^{s}_{p,q}\hookrightarrow B^{s,\dyad}_{p,\infty}$ holds, with $s=1/p-1$.
By definition of the latter space we have the inequality
\Be \label{Haar-as-functional} |\inn{f}{h_{j,\mu} }| \lc 2^{j(\frac 1p-1-s)} \|f\|_{B^{s,\dyad}_{p,\infty}},
\Ee
so the assumed embedding would then imply that $h_{j,\mu}$ defines a bounded linear functional on $B^{-1+1/p}_{p,q}$ (or in the subspace $\mathring{B}^{-1+1/p}_{p,q}$ defined by the closure of $\cS$ in the $B^{-1+1/p}_{p,q}$ norm, in case that $p$ or $q$ are $\infty$). By the duality identities of Besov spaces, see \cite[\S2.11]{Tr83},
this means that $h_{j,\mu} \in B^{1/p'}_{p',q'}$ which cannot be the case if $q'<\infty$, i.e. if $1<q\leq\infty$.


Let $p\le 1$. We use an example from \cite[\S10.1]{GaSeUl21}. We let $\eta_l(x)=2^l \eta(2^l x) $ where $\eta\in C^\infty_c(\bbR)$ is an odd function supported in $(-1/2,1/2)$ such that $\int_0^{1/2} \eta(s) ds=1$ and such that $\int_0^{1/2}\eta(s)s^n ds=0$
for $n=1,2,\dots, M$, for a sufficiently large integer $M$. Let 
\[f_N(x)=\sum_{m=1}^\infty a_m \eta_{N+m} (x-2^{-N+5}m).\] 
By \cite[(85)]{GaSeUl21} we have
\[\|f_N\|_{B_{p,q}^{1/p-1}} \lc\Big(\sum_{m=1}^\infty|a_m|^q\Big)^{1/q}.\]
On the other hand  a calculation shows
$\inn{f_N}{h_{N, 2^5m}} =a_m$ and thus 
\[\|f_N\|_{B^{1/p-1,\dyad}_{p,\infty} } \ge
\Big(\sum_\mu |\inn{f_N}{h_{N,\mu}}|^p\Big)^{1/p} \gc\Big(\sum_{m=1}^\infty |a_m|^p\Big)^{1/p}
\]
which forces $q\le p$.
\end{proof} 

\section{$B^{1/p}_{p,q}$ and $B^{1/p,\dyad}_{p,q}$:  The proof of Theorem \ref{thm:neg} }\label{sec:neg}

\subsection{\it Proof of part (i) of Theorem \ref{thm:neg} } \label{sec:failureofnormequivalence} 
Let \Be\label{defoffN}
f_N=\sum_{j=0}^{N-1}h_{j,0}, \quad\mbox{for $N=1,2,\ldots$}
\Ee
 Observe that 
\Be
\label{fN_dyad}
\|f_N\|_{B^{s, \mathrm{dyad}}_{p,\infty}}=\sup_{j\geq0}2^{j(s-\frac1p)}\big[\sum_{\mu\in\bbZ}|2^j\inn{f_N}{h_{j,\mu}}|^p\big]^{1/p}=1, \quad
\mbox{if $s=1/p$.}
\Ee
On the other hand, using the characterization with differences of order 2 for the $B^s_{p,q}$-norm (since $s\in(0,1]$), see \cite[Theorem 2.5.12]{Tr83}, we have
\Bea
\|f_N\|_{B^{1/p}_{p,\infty}} & \gtrsim & 2^{N/p}\,\big\|\Dt^2_{2^{-N}}(f_N)\big\|_p\nonumber \\
& \geq & 2^{N/p}\,\Big[\int_{-2^{1-N}}^{-2^{-N}}\Big|\sum_{0\leq j<N}\Dt^2_{2^{-N}}h_{j,0}(x)
\Big|^p\,dx\Big]^\frac1p.\label{Dt2}
\Eea
Now a simple computation shows that, if $\dt\in (0, 2^{-j-1}]$, then
\[
h_{j,0}(\cdot+\dt)-h_{j,0}(\cdot) = \bbone_{[-\dt,0)}-2\bbone_{[2^{-j-1}-\dt,2^{-j-1})}+\bbone_{[2^{-j}-\dt,2^{-j})}
\] 
 and therefore, 
 \[
 \Dt_\dt^2h_{j,0}(x)=1, \quad x\in [-2\dt,-\dt).
 \]
 Setting $\delta=2^{-N}$ and inserting this expression into  \eqref{Dt2}
 we get 
\Be
\|f_N\|_{B^{1/p}_{p,\infty}}\gtrsim
 2^{N/p}\,\big\|\Dt^2_{2^{-N}}f_N\big\|_p\geq N,
\label{fNN}
\Ee
and hence part (i) of Theorem \ref{thm:neg} follows.
\qed

\subsection{\it Proof of part (ii) of Theorem \ref{thm:neg} } 
\label{sec:properemb}Consider this time the function 
\[
f=\sum_{j=0}^\infty h_{j,0},
\] and as before $f_N=\sum_{j=0}^{N-1}h_{j,0}$, 
i.e. $f=\lim_{N\to \infty} f_N$ with convergence in $L_p$. Indeed 
\Be
\|f-f_N\|_p\leq\sum_{j\geq N}\big\|h_{j,0}\|_p=\sum_{j\geq N}2^{-j/p}\,\approx\,2^{-N/p}.
\label{ffN}
\Ee
As in \eqref{fN_dyad}, it is again easy to verify that
\Be\label{Bdyadnormone}
\|f\|_{B^{1/p, \text{dyad} }_{p,\infty}}=1.
\Ee
We claim that $f\notin B^{1/p}_{p,\infty}$. Indeed, for large $N$ we have 
\Bea
\|f\|_{B^{1/p}_{p,\infty}} & \gtrsim & 2^{N/p}\,\big\|\Dt^2_{2^{-N}}(f-f_N)+\Dt^2_{2^{-N}}f_N\big\|_p\nonumber \\
& \geq & 2^{N/p}\big( \big\|\Dt^2_{2^{-N}}f_N\big\|_p-4\|f-f_N\|_p\big).\label{aux_fN}
\Eea
Inserting the bounds \eqref{fNN} and \eqref{ffN} into \eqref{aux_fN} gives
\[
\|f\|_{B^{1/p}_{p,\infty}}\,\gtrsim \,2^{N/p}\,\big\|\Dt^2_{2^{-N}}(f_N)\big\|_p-O(1)\,\gtrsim \, N,
\]
which letting $N\nearrow\infty$ proves the assertion.
\qed



\section{Some pathologies of the spaces $B^{s,\dyad}_{p,q}$}
\label{S_Bdyad1}

 We include in this section some pathologies of the spaces $\Bdyad$ when $s>1$, or
$s=1$, $q<\infty$, or $s<1/p-1$, which were mentioned in the introduction.



\subsection{\it Failure of embedding into $B^{s,\dyad}_{p,q}$ for $s>1$ or $s=1$, $q<\infty$}\label{sec:constantfunctions}
The following proposition  is a 
simple result on the theme of Brezis' paper \cite{brezis-constant} on how to recognize constant  functions. 

\begin{proposition}\label{P_Bdyad1}
Let $0<p,q\le \infty$ and assume that either  (i) $s>1$,  or  (ii) $s=1$ and $q<\infty$.  

Then every $f\in C^1(\bbR) \cap B_{p,q}^{s,\dyad}(\bbR)$ is a constant function. 
\end{proposition} 

\begin{Remark} Bo\v ckarev's results \cite[Theorem 3]{Bo69} indicate that less restrictive assumptions can be made but we will not pursue the problem of optimal hypotheses here.
\end{Remark}



\begin{proof} 
We argue as in the proof of Lemma \ref{lem:ex1_B}, now using Taylor's formula  in the form 
$$\sup_{b\in K}\sup_{|b-y|\le \eps}  | f(y)- f(b)  - f'(b) (y-b) | = o(\eps) $$
for any compact $K$. Take  $b\equiv b_{j,\mu}= 2^{-j} (\mu+\frac 12)$  to see that
\Be \label{eq:taylor}\inn{f}{h_{j,\mu} }= f'(2^{-j}(\mu+\tfrac 12)) 2^{-2j-2} +o(2^{-2j})\Ee
with uniformity in the remainder as $b_{j,\mu}$ ranges over a compact set. 

Now assume  that  $f\in C^1$ and that $f'$ is not identically zero.  Then there is a dyadic  interval $J=[\nu2^{-\ell}, (\nu+1) 2^{-\ell} )$ and $c>0$ such that for $j\ge j_0>\ell$  
\[ |\inn{f}{h_{j,\mu} }| \ge c2^{-2j}, \quad\mbox{if $I_{j,\mu}\subset J$}.\]
Hence 
\Beas
\|f\|_{\Bdyad} & \geq & \Big(\sum_{j=j_0}^\infty \Big[2^{j(s-\frac1p)}\Big(\sum_{\mu: I_{j,\mu}\subset J} |2^j\hdot{f,h_{j,\mu}}|^p\Big)^\frac1p
\Big]^q\Big)^\frac1q\\
& \ge  & \Big(\sum_{j=j_0}^\infty \Big[2^{j(s-{1/p})} 2^{(j-\ell)/p} c2^{-j}\Big]^q\Big)^\frac1q \ge c_\ell \Big(\sum_{j=j_0}^\infty 2^{j(s-1)q}\Big)^{1/q} 
\Eeas
with $c_\ell>0$. Hence $\|f\|_{B^{s,\dyad}_{p,q}}=\infty$ when $s>1$ or when $s=1$ and $q<\infty$. 

We conclude that for this range we have $f'\equiv 0$  for every $f\in C^1\cap B^{s,\dyad}_{p,q}$ and 
Proposition \ref{P_Bdyad1}  follows. 
\end{proof}

\subsection{\it The dyadic Besov-spaces for $s<1/p-1$: Failure of completeness}

\begin{proposition}\label{P_Bdyad2}
Let $0<p,q\leq\infty$. If $s<1/p-1$ then the spaces $\Bdyad(\SR)$ are not complete.
\end{proposition}
\begin{proof}
Consider the functions 
\Be
f_N=\bbone_{[0,1)}+\sum_{j=0}^{N-1}2^jh_{j,0}=2^N\bbone_{[0,2^{-N})}, \quad N=1,2,\ldots
\label{fNx}
\Ee
It is easily seen that, under the assumption $s<1/p-1$, then
\[
\|f_{M}-f_N\|_{\Bdyad}=\Big(\sum_{N\leq j<M} [2^{j(s+1-\frac1p)}]^q\Big)^\frac1q\to 0,
\]
when $M>N\to\infty$. So, $\{f_N\}_{N\geq1}$ is a Cauchy sequence in $\Bdyad$. 
However, the distributional limit of $f_N$ is the Dirac measure $\dt$, which does not belong to the
space $\sB$.
\end{proof}

\subsection{\it Failure of an embedding for $s=1/p-1$}
A small variation of the last example shows also part (i) of Theorem \ref{emb_1/p-1} and at the same time the optimality of the condition $s>1/p-1$ in part (ii) of Proposition \ref{prop1}  when $q<\infty$.

\begin{proposition}\label{P_ex3}
Let $0<p, u\leq\infty$. Then
\[
B^{1/p-1}_{p,u}\not\hookrightarrow B^{1/p-1, \rm{dyad}}_{p,q}, \quad 0<q<\infty.
\]
\end{proposition}
\begin{proof} 
Consider $f_N$ as in \eqref{fNx}, and let $g_N=f_N-f_N(-\cdot)$ be its odd extension.
Then, it was shown in \cite[Proposition 52]{GaSeUl21} (Proposition 13.3 in arXiv:1901:09117)
that
\[
\|g_N\|_{B^{1/p-1}_{p,u}}\lesssim 1, 
\]
for all $0<p,u\leq \infty$. However, it is easily seen that
\[
\|g_N\|_{B^{1/p-1, \rm{dyad}}_{p,q}}\geq \|f_N\|_{B^{1/p-1, \rm{dyad}}_{p,q}}\geq N^{1/q}.
\qedhere
\]
\end{proof}
\begin{remark}
\label{R_ex3}
Note that the above proof also shows that $L_1$ is not continuously embedded into $\sB=B^{-1}_{\infty,1}$. Indeed, the functions $f_N=2^N\bbone_{[0,2^{-N})}$ in \eqref{fNx} satisfy $\|f_N\|_1=1$ and $\|f_N\|_{\sB}\gtrsim N$.
\end{remark}


\bibliographystyle{abbrv}

\end{document}